\documentclass[11pt,a4paper]{amsart}
\usepackage{mathrsfs}
\usepackage{amsfonts}
\usepackage[active]{srcltx}
\usepackage{enumerate}
\usepackage{color}
\usepackage{hyperref}

\usepackage{amsmath,amssymb,xspace,amsthm}

\theoremstyle{definition}
\newtheorem{thm}{Theorem}[section]
\newtheorem{lem}[thm]{Lemma}

\newtheorem{cor}[thm]{Corollary}
\newtheorem{prop}[thm]{Proposition}
\newtheorem{rem}[thm]{Remark}
\newtheorem{exa}[thm]{Example}
\numberwithin{equation}{section}

\def\htt{\mathrm{ht}}\def\Ht{\mathrm{Ht}}

\def\Ind{\mathrm{Ind}}

\def\S{\mathcal{S}}

\def\T{\mathcal{T}}

\newcommand{\C}{\mathbb C}

\newcommand{\N}{\mathbb{N}}
\newcommand{\Z}{\mathbb{Z}}

\renewcommand{\l}{\lambda}
\newcommand{\si}{\sigma}
\newcommand{\p}{\partial}

\newcommand{\ga}{\gamma}

\newcommand{\vphi}{\varphi}

\def\ot{\otimes}
\def\O{\Omega}

\def\LL{\mathcal{L}}

\def\GG{\mathcal{G}}

\def\g{\mathfrak{g}}\def\ga{\gamma}\def\Ga{\Gamma}
\def\tg{\widetilde{\g}}
\def\hg{\widehat{\g}}
\def\tM{\widetilde{M}}
\def\hM{\widehat{M}}\def\bM{\overline{M}}\def\bu{\bar{u}}\def\bw{\bar{w}}

\def\h{\mathfrak{h}}

\def\n{\mathfrak{n}}

\def\sl{\mathfrak{sl}}

\def\hsl{\widehat{\mathfrak{sl}}}

\def\normOrd #1{{\mathop{:}\nolimits\!#1\!\mathop{:}\nolimits}}

\newcommand{\bo}{\pmb{0}}

\title[Smooth modules]{\bf Smooth representations of  affine Kac-Moody algebras}
\author{V. Futorny,  X. Guo, Y. Xue and K. Zhao}

\date{}

\begin{document}
	
	\maketitle
	
	\begin{abstract}
		Smooth modules for affine Kac-Moody algebras have a prime importance for the quantum field theory as they correspond to the representations of the universal affine vertex algebras. But, very little is known about such modules beyond the category of positive energy representations.
		We construct a new class of  smooth modules over   affine Kac-Moody algebras. In a particular case, these modules are isomorphic to those induced from generalized Whittaker modules for  Takiff Lie algebras.
		  We establish the irreducibility criterion for constructed modules in the case of the Lie algebra $A^{(1)}_1$.
		\vskip.5cm
		
		\noindent {\bf Keywords:}  Affine Kac-Moody  algebras, irreducible  modules, \\
        smooth modules, Takiff algebras

		\noindent{\bf Mathematics Subject Classification} 2010: 17B10,
		17B67, 17B69.
	\end{abstract}
	
	
	\section{Introduction}
	
	Among infinite dimensional Lie algebras the affine Kac-Moody algebras have  a very important placement due to their  numerous
	 applications in different areas of  Mathematics and Physics:
	 string theory and conformal fields theory \cite{FMS}), vertex theory \cite{LL}, quantum groups \cite{CP, Lu}, representations of associative algebras and combinatorics \cite{K}, to name just a few.

	The representation theory of affine Lie algebras is quite different in many aspects to that of finite dimensional
	simple Lie algebras. For example, there are irreducible modules  over  the  affine Kac-Moody algebras, including the ones we construct in this paper,  that do not have  their counterparts in the finite dimensional case.
	
	Smooth representations are of a special significance due to their relevance to vertex algebras and conformal field theory. Any smooth representation of an affine Kac-Moody algebra is a representation of the corresponding universal affine vertex algebra.
	The simplest family of smooth modules consists of integrable highest weight modules with locally nilpotent action of real root elements of the Lie algebra.
	Integrable highest  weight modules was the first class of
	representations over affine Lie algebras extensively studied \cite{K}. All irreducible integrable weight modules
	with finite dimensional weight spaces over  untwisted affine Lie
	algebras were classified by
	Chari \cite{Ch}.  This classification was extended  to the twisted case by  Chari and  Pressley \cite{CP2}. Every irreducible integrable weight module with finite dimensional weight spaces is
	either  highest weight module or a loop module (not smooth), plus a spectral flow.  Loop modules with finite-dimensional weight spaces were studies in \cite{E}.
	
	  A comprehensive program studying
	weight modules for arbitrary affine Lie algebras was initiated in the 1990's \cite{F3}, and a classification of all irreducible weight
	modules for affine Lie
	algebras with nonzero level and with finite dimensional weight spaces  was obtained in \cite{FT}. Such modules
	are induced from cuspidal modules over Levi subalgebras of type $A$ and $C$  of proper parabolic subalgebras, that is they are positive energy representations, and hence smooth. If the level is zero, then the classification of such irreducible modules is claimed in an unpublished paper \cite{DG}.
	
	 First examples of irreducible weight modules  with infinite dimensional
	weight spaces were constructed by Chari and Pressley in \cite{CP3} by taking a tensor product of an irreducible integrable highest weight module and an integrable loop module.
	 Theory of Verma type modules and generaized Verma type was  developed in \cite{CFu}, \cite{F1}, \cite{F2}, \cite{F3}, \cite{BBFK}, \cite{FK1}, \cite{FK2}. General theory of positive energy representations was developed in \cite{FKr}.
	 A complete classification
	of all irreducible (weight and non-weight) modules over affine Lie
	algebras with locally nilpotent action of the nilpotent radical was
	obtained in \cite{MZ}.
	
	 A famous class of non-weight modules over  affine Kac-Moody
	algebras consists of Whittaker modules. Their
	structure in the case of  $A_1^{(1)}$ was systematically studied in \cite{ALZ}, while this
	study for other affine Lie algebras is still open.
	
	A class of imaginary Whittaker modules was introduced in \cite{Chr}, and corresponding induced modules were studied in \cite{CF}.
	
	Unfortunately, all these results did
	 not provide new examples of smooth modules. In fact, very little is known about  smooth weight modules  with infinite dimensional
	weight spaces or smooth non-weight modules. The purpose of this paper if to construct such modules for an arbitrary untwisted affine Kac-Moody algebra.


Let  $\g$ be a simple finite-dimensional Lie algebra with a Cartan decomposition 	$\g=\n_-\oplus \h\oplus \n_+$, $\hg=\LL(\g) \oplus \C K$ and $\tg= \hg \oplus \C d$, where
$\LL(\g)$ is the loop algebra of $\g$, $K$ is the central element and $d$ is the degree derivation.

Fix a set pairs of integers associated with simple roots: 	$$\Sigma=\{(N_{\alpha}, M_{\alpha}), \, N_{\alpha}+M_{\alpha}\geq 0, \,  \alpha\in \Pi\}.$$ Denote by $\S_{\Sigma}$ the Lie subalgebra of $\hg$ generated by $K$ and the elements of the form
	$$h_{\alpha}(n_{\alpha,0}), \, e_{\alpha}(n_{\alpha}), \, f_{\alpha}(m_{\alpha}),$$ where
	$ n_{\alpha,0}, n_{\alpha}, m_{\alpha}\in\Z, n_{\alpha,0}\geq 0, n_{\alpha}>N_{\alpha}, m_{\alpha}>M_{\alpha}$,
	$\alpha\in \Pi$. Consider a Lie algebra homomorphism $\varphi_{\Sigma}: \S_{\Sigma}\to\C$, the induced smooth $\hat\g$-module	
$\hM(\varphi_{\Sigma})=\Ind_{\S_{\Sigma}}^{\hg}\C v$ and the induced smooth $\tg$-module
	$\tM(\varphi_{\Sigma})=\Ind_{\S_{\Sigma}}^{\tg}\C v,$ where $\S_{\Sigma}$ acts  on $v$ with respect to  $\varphi_{\Sigma}$. If $N_{\alpha}=n$ and $M_{\alpha}=m$ for all $\alpha\in \Pi$, we simply write $\hM(\varphi_{n, m}) $ and $\tM(\varphi_{n, m})$ for corresponding induced modules.
	Modules $\hM(\varphi_{n, m}) $ have a combined weight-Whittaker structure and they
	 can be realized as an induced  modules over
 the Takiff algebra.

Let	$\T_m(\g)=\g\ot(\C[t]/t^{m+1}\C[t])$ be the $m$-th generalized Takiff algebra with triangualr decomposition
 $ \T_m(\g)=\T_m(\g)_+\oplus\T_m(\g)_0\oplus\T_m(\g)_-,$
	where $\T_m(\g)_{\pm}=\n_{\pm}\ot(\C[t]/t^{m+1}\C[t])$	and $\T_m(\g)_{0}=\h\ot(\C[t]/t^{m+1}\C[t])$. Taking an
	arbitrary Lie algebra homomorphism $\psi: \T_m(\g)_{0}\to\C$, we define
	a $\T_m(\g)_+\oplus\T_m(\g)_0$-module structure on $\C$ with a trivial action of $\T_m(\g)_{+}$, and construct  the induced $\T_m(\g)$-module
	$V(\psi)$.  Using the canonical surjective homomorphism $\g\ot\C[t]\to\T_m(\g)$,
	  we can view $V(\psi)$ as a module over $\g\ot\C[t]\oplus\C K$, on which $K$ acts as a multiplication by $\theta\in \C$. Finally, we define
	 the induced $\hg$-module
	$M(\psi,\theta)=\Ind_{\g\ot\C[t]\oplus\C K}^{\hg}V(\psi).$
	
It turns out that $M(\psi, \theta)\cong\hM(\vphi),$	
for some homomorphism $\vphi:\S_{-1,m} \to \C$ such that $\vphi(K)=\theta$.

 We address the irreducibility of $\hM(\vphi)$ in the case of the affine Lie algebra $\hsl_2$.
 Our main results are summarized in the following theorems (cf. Theorem \ref{irre general} and Theorem \ref{irre K=-2}).

 \begin{thm} Let $\hg=\hsl_2$.
Suppose $n, m\in\Z$ with $n+m\geq0$ and $\vphi$ is a Lie algebra homomorphism from $\S_{n, m}$ to $\C$.
The $\hg$-module $\hM(\vphi)$ is irreducible if and only if $\vphi(h(n+m+1))\neq0$
and $\vphi(K)\neq-2$.
\end{thm}

Let  $\vphi(K)=-2$. In this case the module $\hM(\vphi)$ is not irreducible. Every set of numbers $\{\theta_i\in\C, \, i\in\N\}$ parametrizes a certain submodule $\hM(\theta_i, i\in\N)$ of $\hM(\vphi)$ and we have

\begin{thm}
 Suppose $\vphi(h(N))\neq0$ and
$\vphi(K)=-2$. Then the submodule $\hM(\theta_i, i\in\N)$ is maximal and hence $\hM(\vphi)/\hM(\theta_i, i\in\N)$ is irreducible for any
$\theta_i\in\C$. Moreover, any  irreducible $\hsl_2$-subquotient of $\hM(\vphi)$
is isomorphic to $\hM(\vphi)/\hM(\theta_i, i\in\N)$ for some choice of $\{\theta_i\in\C, \, i\in\N\}$.

\end{thm}	
	
We also establish the isomorphisms between modules $\hM(\vphi)$ and their irreducible quotients (cf. Theorem \ref{iso hg}). Analogous results are obtained for the $\tilde\g$-modules (cf. Theorem \ref{irre tM} and Theorem \ref{irre tM K=-2}).

	{Throughout this paper, we denote by $\Z$, $\N$, $\Z_+$ and $\C$ the
		sets of integers, positive integers, nonnegative integers and
		complex numbers respectively. All vector spaces and Lie algebras are
		over $\C$.} For a Lie algebra $\GG$, we denote its universal
	enveloping algebra by $U(\GG)$. 

	\section{Affine Lie algebras and smooth modules}
	
\subsection{Affine Lie algebras}	
For a simple finite-dimensional Lie algebra $\g$ we consider
   the loop algebra $\LL(\g):=\g\ot\C[t^{\pm1}]$ and its universal central extension $\hg:=
	\LL(\g) \oplus \C K$
  with the Lie algebra structure determined by
	\begin{equation*}\aligned
		& [x \otimes t^m, y \otimes t^n ] = [x,y] \otimes t^{m+n} + (x|y) m \delta_{m+n,0} K,\\	
		\endaligned\end{equation*}
 for $x,y\in \g$ and $[K, \hg]=0,$ where $(x|y)$ is the normalized  non-degenerate symmetric bilinear form of $\g$.
	Let $d$ be the degree derivation $t\frac{d}{dt}$ acting on $\LL(\g)$.
	Then
	\begin{equation*}
		\tg:= \LL(\g) \oplus \C K \oplus \C d
	\end{equation*}
	possesses a natural semidirect sum Lie algebra structure
	 determined by $ [ d, x \otimes t^m ] = m x \otimes t^m,$ $[K, d]=0.$
	The algebras $\tg$ and $\hg$ are  the   \emph{untwisted}  affine Lie algebras of $\g$.

 Regard $\g$ as a subalgebra of $\hg\subseteq\tg$ by identifying $x\ot 1=x$ for any $x\in\g$, and denote $x(n)=x\otimes t^n$.
	A $\hg$-module or $\tg$-module is called of level $\kappa\in\C$ if the central element $K$ acts
	on the module as multiplication by the scalar $\kappa$.
	
	Both the algebras $\hg$ and $\tg$ have natural $\Z$-gradings: $$\hg=\bigoplus_{i\in\Z}\hg_i, \text{ and } \tg=\bigoplus_{i\in\Z}\tg_i,$$ where $\hg_i=\tg_i=\g\ot t^i$ for $i\neq0$,
	$\hg_0=(\g\ot t^0)\oplus\C K$ and $\tg_0=(\g\ot t^0)\oplus\C K\oplus\C d$.
	For any $n\in\Z_+$, denote $\hg_{\geq n}=\bigoplus_{i\geq n}\hg_i$.
	A module $M$ over $\hg$ or $\tg$ is called a {\bf smooth module} if for any $v\in M$, there
	exists $n\in\N$ such that $\hg_{\geq n}v=0$.

  Let $\h$ be a Cartan subalgebra of $\g$ and $\Delta$ the root system of $\g$ with respect to $\h$.
  Fix the set of simple roots $\Pi \subset \Delta$ and denote by
   $\Delta_+$ the positive roots with respect to $\Pi$. For $\alpha \in \Delta_+$, let $h_\alpha \in \h$ be the corresponding coroot and let $e_\alpha$ and $f_\alpha$ be basis of root subspaces $\g_\alpha$ and $\g_{-\alpha}$, respectively. We assume that  $[e_\alpha, f_\alpha] = h_\alpha$. Let  $\Delta^{im}=\{k\delta|
k\in\Z\setminus\{0\}\}$ be the set of imaginary roots of $\hg$, where $\hg_{\delta}=\h\otimes t$.

   Any $\hg$-module $V$ can be made into a $\tg$-module by taking a tensor product $\C[d]\ot V$, with a natural action of $\tg$. The action of $d$ is free on $\C[d]\ot V$. In its turn, every $\tg$-module can be   regarded as a $\hg$-module by the restriction.

 \medskip

 \subsection{Affine vertex algebras and Segal-Sugawara construction.}\label{affineVA}
	
	Now let us recall some basic notation of affine vertex algebras.
	For any $\kappa\in\C$, let $\C_\kappa$ be the $1$-dimensional module over $\hg_{\geq0}\oplus\C K$
	such that $\hg_{\geq0}\C_\kappa=0$ and $K$ acts as the multiplication by $\kappa$.
	The generalized Verma module
	$$V_\kappa(\g)=U(\hg)\otimes_{U(\hg_{\geq0}\oplus\C K)}\C_\kappa$$
	admits a natural vertex algebra structure as follows. For an element $x \in \g$, we denote by $x(z) \in \hg[[z^{\pm 1}]]$ the formal distribution defined by
\begin{align*}
x(z) = \sum_{n \in \Z} x(n) z^{-n-1}.  \label{eq:field a}
\end{align*}
The commutation relations  for $\hg$ become
\begin{align*}
  [x(z),y(w)] = [x,y](w)\delta(z-w) + (x|y)K\partial_w \delta(z-w),
\end{align*}
for $x, y \in \hg$. The state-field correspondence $Y \colon {V}_\kappa(\g) \rightarrow End \, {V}_\kappa(\g)[[z^{\pm 1}]]$ is given by
\begin{align*}
  Y(x_{1}(-n_1-1) \dots x_{k}(-n_k-1){\bf 1}, z) = {1 \over n_1! \dots n_k!}\, \normOrd{\partial_z^{n_1} x_1(z) \dots \partial_z^{n_k}x_k(z)}
\end{align*}
for $k \in \N$, $n_1,n_2,\dots,n_k \in \N_0$ and $x_1,x_2,\dots, x_k \in \g$, where ${\bf 1}=1\otimes 1\in V_\kappa(\g)$ is the vacuum vector. Here $\normOrd{\,\,}$ denotes the standard normal ordered product of  formal distributions.

We have
	$$x(z)=Y(x(-1){\bf 1}, z)=\sum_{n\in\Z}x(n)z^{-n-1},\  x\in\g,$$
	for $x\in \g$.
	It is well known that  smooth modules over $\hg$ of level $\kappa$ are
	exactly the module over the vertex algebra $V_\kappa(\g)$.
	
	Recall  the well-known Segal-Sugawara construction of a conformal element in $V_\kappa(\g)$.
	Assume $\kappa\neq-2$ and
	denote
	$$\omega=\frac{1}{2(\kappa+2)}\sum_i \normOrd{x_i (z)x^*_i(z)}{\bf1},$$
	where $x_i, x_i^*$ are dual bases of $\g$.
	Then we have
	$$Y(\omega, z):=L(z)=\sum_{n\in\Z}L(n)z^{-n-2},$$
	and the operators $L(n)\in U(\hg), n\in\Z$ satisfy the Virasoro relations
	$$[L(m),x(n)]=-nx(m+n),\quad m,n\in\Z,$$
	and
	$$[L(m),L(n)]=(m-n)L(m+n)+\frac{\kappa(m^3-m)}{4(\kappa+2)}\delta_{m+n,0},\quad m,n\in\Z.$$

\medskip

  \subsection{Construction of smooth modules}

  A well known class of smooth $\hg$-modules is given by the Whittaker modules. In this section we introduce a weak version of these modules, which surprisingly leads to new
  families of smooth $\hg$-modules.

   For every simple root $\alpha\in \Pi$ fix a pair of integers $N_{\alpha}, M_{\alpha}$ with
	 $N_{\alpha}+M_{\alpha}\geq 0$. Denote by $\Sigma$ the set $\{t_{\alpha}=(N_{\alpha}, M_{\alpha}), \, \alpha\in \Pi\}$.
	Let $\S_{\Sigma}$ be the Lie subalgebra of $\hg$ generated by $K$ and the elements of the form
	$$h_{\alpha}(n_{\alpha,0}), \, e_{\alpha}(n_{\alpha}), \, f_{\alpha}(m_{\alpha}),$$ where
	$ n_{\alpha,0}, n_{\alpha}, m_{\alpha}\in\Z, n_{\alpha,0}\geq 0, n_{\alpha}>N_{\alpha}, m_{\alpha}>M_{\alpha}$,
	$\alpha\in \Pi$.
	
	Consider any Lie algebra homomorphism $\varphi_{\Sigma}: \S_{\Sigma}\to\C$. Then
	$$\vphi_{\Sigma}(e_{\alpha}(n_{\alpha}))=\vphi_{\Sigma}(f_{\alpha}(m_{\alpha}))=0,$$
	for any $\alpha\in \Pi$ and all $n_{\alpha}>N_{\alpha}$, $m_{\alpha}>M_{\alpha}$. Moreover, $\vphi_{\Sigma}(h(n_{\alpha,0}))=0$ for $n_{\alpha,0}>N_{\alpha}+M_{\alpha}+1\geq1$. Also, if $\beta=\alpha_1+\ldots +\alpha_k$
	with $\alpha_i\in \Pi$, $i=1, \ldots, k$, then $\vphi_{\Sigma}(\hg_{\beta+n_{\beta}\delta})=0$ for any $n_{\beta}>\sum_{i=1}^k N_{\alpha_i}+k$. Similarly, $\vphi_{\Sigma}(\hg_{-\beta+n_{-\beta}\delta})=0$ for any $n_{-\beta}>\sum_{i=1}^k M_{\alpha_i}+k$.
	
Define a $1$-dimensional $\S_{\Sigma}$-module $\C_\Sigma=\C v$ by
	$$xv=\varphi_{\Sigma}(x)v, \forall x\in\S_{\Sigma}.$$
	Then we can form the induced $\hg$-module
	$$\hM(\varphi_{\Sigma})=\Ind_{\S_{\Sigma}}^{\hg}\C_\Sigma=U(\hg)\ot_{U(\S_{\Sigma})}\C v.$$
	and the induced $\tg$-module
	$$\tM(\varphi_{\Sigma})=\Ind_{\S_{\Sigma}}^{\tg}\C_\Sigma=U(\tg)\ot_{U(\S_{\Sigma})}\C v=\C[d]\ot\hM(\vphi_{\Sigma}).$$
	
	It is worthwhile to mention that $\hM(\vphi_{\Sigma})$ is a weight module with respect to the action of
	$\h=\h\ot t^0$. We will use the terminologies such as weights, weight vectors and so on freely.
		
		Note that both $\hM(\vphi_{\Sigma})$ and $\tM(\vphi_{\Sigma})$ are smooth modules
	over $\hg$ and $\tg$ respectively. A priori the structure of these modules is absolutely not clear.

   We will simplify the construction above as follows.   Fix a pair of integers $N$ and $M$ with
	 $N+M\geq 0$.
	Let $\S_{N,M}$ be the Lie subalgebra of $\hg$ generated by $K$ and the elements of the form
	$$h(k), \, x_{\beta}(s), \, y_{\beta}(r),$$ for $h\in \h$, $x_{\beta}\in \g_{\beta}, $ $y_{\beta}\in \g_{-\beta}$, $\beta\in \Delta_+$,
	$ k, s, r\in\Z$, $k\geq 0, s>N, r>M$.

 Now take an arbitrary Lie algebra homomorphism $\varphi_{N,M}: \S_{N,M}\to\C$. Then
	$$\vphi_{N,M}(x_{\beta}(n))=\vphi_{N,M}(y_{\beta}(m))=0,$$
	for any $\beta\in \Delta_+$ and all $n>N$, $m>M$. Also, $\vphi_{N,M}(h(k))=0$ for $k>N+M+1\geq1$.

 For a $1$-dimensional $\S_{N,M}$-module $\C v$ with
	$xv=\varphi_{N,M}(x)v$,  $x\in\S_{N,M}$
	the induced $\hg$-module
	$\hM(\varphi_{N,M})=U(\hg)\ot_{U(\S_{N,M})}\C v$
	and the induced $\tg$-module
	$\tM(\varphi_{N,M})=\C[d]\ot\hM(\vphi_{N,M})$ are smooth modules.

   \medskip

   \subsection{Connection to Takiff modules}
   Let $\g$ be a simple finite-dimensional Lie algebra with a Cartan decomposition
   $$\g=\n_-\oplus \h\oplus \n_+.$$
   For any positive integer $N$, consider the Takiff algebra
	$$\T_N(\g)=(\g\ot\C[t])/\g\ot t^{N+1}\C[t]=\g\ot(\C[t]/t^{N+1}\C[t]).$$
	Then we have a triangular decomposition of $\T_N(\g)$ defined as follows
	$$\T_N(\g)=\T_N(\g)_+\oplus\T_N(\g)_0\oplus\T_N(\g)_-,$$
	where $\T_N(\g)_+=\n_+\ot(\C[t]/t^{N+1}\C[t])$, $\T_N(\g)_{-}=\n_-\ot(\C[t]/t^{N+1}\C[t])$
	and $\T_N(\g)_{0}=\h\ot(\C[t]/t^{N+1}\C[t])$.
	For any Lie algebra homomorphism $\psi: \T_N(\g)_{0}\to\C$ define a $1$-dimensional module $\C_\psi=\C u$ over $\T_N(\g)_+\oplus\T_N(\g)_0$ as follows:
	$$(x+y)u=\psi(y)u,\ \forall\ x\in\T_N(\g)_+ \text{ and } y\in\T_N(\g)_0.$$
	Then we form the induced $\T_N(\g)$-module
	$$V(\psi)=\Ind_{\T_N(\g)_+\oplus\T_N(\g)_0}^{\T_N(\g)}\C_\psi.$$
	The module  $V(\psi)$ is a highest weight module with highest vector $u$ with respect to the Cartan subalgebra $\h=\h\ot\C$. The irreducibility of $V(\psi)$ was determined in \cite{Wi}:
	
	\begin{thm}\label{Wilson}
		The $\T_N(\g)$-module $V(\psi)$ is irreducible if and only if $\psi(h_{\alpha}\ot t^N)\neq0$, for all $\alpha\in \Delta_+$.
	\end{thm}
	
	Now  we can view $V(\psi)$ as a module over $\g\ot\C[t]\oplus\C K$ by a pullback through the canonical epimorphism
	$\g\ot\C[t]\to\T_N(\g)$ and requiring that $K$ acts as multiplication by a scalar $\theta$.
	Then we form the induced $\hg$-module
	$$M(\psi,\theta)=\Ind_{\g\ot\C[t]\oplus\C K}^{\hg}V(\psi).$$
	This module is also a weight module with respect to $\h=\h\ot\C$.

 Denote $\S(N):=\S_{-1,N}$ and define the homomorphism $\vphi_N=\vphi_{-1,N}:\S(N) \to \C$ as follows. Set $\vphi_N(K)=\theta$,
 $$\vphi_N(h(k))=\vphi_N(x_{\beta}(s))=\vphi_N(y_{\beta}(r))=0,$$
 for $s\geq 0$, $k, r>N$, $\beta\in \Delta_+$. Let $\p: \C h\ot\C[t]\to\T_N(\g)_0$ be  the canonical projection. We set $\vphi_N(h(k))=\psi(\p(h(k)))$ for $0\leq k\leq N$.

\medskip

\begin{lem}\label{lem-Takiff}
We have the isomorphism of $\hg$-modules:
$$M(\psi, \theta)\cong\hM(\vphi_N).$$
\end{lem}

\begin{proof} Since
 $xu=\vphi_N(x)u$ for all $x\in \S(N)$, we have a homomorphism from $\hM(\vphi_N)$ to $M(\psi, \theta)$ by the universality property. The isomorphism follows from the PBW theorem.
\end{proof}	

The isomorphism in Lemma \ref{lem-Takiff} will be crucial in the proof of irreducibility
of constructed modules.

\begin{rem}
Let $\g=sl(2)$ with a simple root $\alpha$ and $N_{\alpha}=M_{\alpha}=n-1$ for some $n\in\N$. If $\vphi$ is a Lie algebra homomorphism from $\S_{n-1,n-1}$ to $\C$ with $\vphi(K)=0$, then
	the $\hg$-module $\hM(\vphi)$
	and the $\tg$-module $\tM(\vphi)$ are just the modules considered in \cite{NXZ}.
\end{rem}

   \medskip
	
	\subsection{Affine Lie algebra  $A^{(1)}_1$}
	
	Now we assume that $\g=\sl_2$ with a standard basis $\{e,f,h\}$ and the bilinear form normalized as
$(e|f)=1$, $(h|h)=2.$ Since the Cartan subalgebra $\h=\C h$ is $1$-dimensional,
 we will identity weights with respect to $\h$ with the eigenvalues of $h$.
	
	Consider the Casimir operator $\Omega$ of $ \widetilde{\sl}_2$, which can be written explicitly as
	(cf. \cite{K} or \cite{ALZ})
	\begin{equation}\aligned\label{casimir}
		\Omega =& 2dK + 4d +\frac{1}{2}h(0)^2 +h(0) +2f(0)e(0)\\
		&+2\sum_{n=1}^{\infty} (e(-n)f(n) +f(-n)e(n) +\frac{1}{2}h(-n)h(n)).
		\endaligned\end{equation}
	The action of $\Omega$ is well-defined on any smooth $\tg$-module,
	and also on any smooth $\hg$-module of level $-2$.
	Moreover, it is well known that if the action of $\Omega$ is well-defined  on a  $\tg$-module (respectively  a $\hg$-module of level $-2$), then it commutes with the action  of $U(\tg)$ (respectively  $U(\hg)$).

If $\kappa\neq-2$, then
	 $$\omega=\frac{1}{4(\kappa+2)}\big(2e(-1)f(-1)+2f(-1)e(-1)+h(-1)^2\big){\bf1}$$
	 is a conformal vector in $V_\kappa(\g)$.
	Suppose $\kappa=-2$. Denote $$\tau=\frac{1}{4}\big(2e(-1)f(-1)+2f(-1)e(-1)+h(-1)^2\big){\bf1}.$$
	Then
	$$Y(\tau,z)=T(z)=\sum_{n\in\Z}T(n)z^{-n-2},$$
	where the operators $T(n)\in U(\hg), n\in\Z$ satisfy the commutation relation
	\begin{equation}\label{T(n)}
		[T(n),x(m)]=0,\quad m,n\in\Z.
	\end{equation}
	In particular, all $T(n), n\in\Z$ commute with all elements in $U(\hg)$.
	It is well known from  \cite{LL},(6.2.44) that
	\begin{equation}\label{T(n)'}
		T(n)=\frac{1}{4}\sum_{i\in\Z}\big(2:e(-i)f(n+i):+2:f(-i)e(n+i):+:h(-i)h(n+i):\big),
	\end{equation}
	where the normal order $:e(i)f(j):$ is defined as follows:
	$$:e(i)f(j):=\left\{
	\begin{array}{cc}
		e(i)f(j),&\mbox{if}\,\, i<0,\\
		f(j)e(i), & \mbox{if}\,\, i\ge0.\\
	\end{array}\right.
$$
The normal order for the other two terms is defined similarly.

Now we define our modules.
	For any integers $N_1, N_2$ with $N_1+N_2\geq0$,
the subalgebra $\S_{N_1,N_2}$ of $\hg$ is  spanned by $K$ and the elements
	$$h(n_0), e(n_1), f(n_2),\text{ where } n_0,n_1,n_2\in\Z, n_0\geq0, n_1>N_1, n_2>N_2.$$
	If $\varphi_{N_1, N_2}: \S_{N_1,N_2}\to\C$ is a Lie algebra homomorphism and $\C_\varphi=\C v$ is a $1$-dimensional $\S_{N_1,N_2}$-module with
	$$xv=\varphi(x)v, \forall x\in\S_{N_1,N_2},$$
	then we have the induced modules $\hM(\varphi_{N_1,N_2})$ and $\tM(\varphi_{N_1,N_2})$.

	A nonzero element $w\in\hM(\vphi_{N_1,N_2})$ is a {\bf singular vector} of weight $\lambda\in\C$ if
	$$e(N_1+i)w=f(N_2+i)w=(h(i)-\vphi(h(i)))w=0$$
	for any $i\in\N$ and $hw=\l w$.
	For example, $v$ is a singular vector of weight $\vphi(h(0))$.
	If  the weight of a singular vector is clear, we simply call it a singular vector.
	
	Since $\hM(\vphi_{N_1,N_2})$ and $\tM(\vphi_{N_1,N_2})$ are smooth modules
	over $\hg$ and $\tg$ respectively, we have a well-defined action of $\Omega$ on $\tM(\vphi_{N_1,N_2})$, and also on $\hM(\vphi_{N_1,N_2})$ if $\vphi(K)=-2$.
	%
	%
%

By Lemma \ref{lem-Takiff}, we have an isomorphism of $\hg$-modules $M(\psi,\theta)\cong\hM(\vphi_N)$ for any Lie algebra homomorphism $\psi: \T_N(\g)_{0}\to\C$ and $\theta\in \C$,
	where the Lie algebra homomorphism $\vphi_N:=\vphi_{-1,N}: \S_{-1,N}\to\C$
	is given by
	$$\vphi_N(e(n_1))=\vphi_N(f(n_2))=0,\ \vphi_N(h(n_0))=\psi(\overline{h(n_0)}),\ \vphi_N(K)=\theta,$$
	for any $n_0, n_1, n_2\in\Z_+, n_2>N$. Here $\overline{h(n_0)}$ is the image of $h(n_0)$ under the canonical projection $\C h\ot\C[t]\to\T_N(\g)_0$.
Using this isomorphism we will identify both modules and denote the generator of
$\hM(\vphi)$ by $u$.

\section{Irreducibility and isomorphism}

From now on we will always assume that $\hg= \widehat{\sl}_2$ and $\tg= \widetilde{\sl}_2$.

\subsection{Reduction to the case $N_1=-1$.}

Fix   integers $N_1, N_2$ with $N_1+N_2\geq0$ and a  Lie algebra homomorphism
$\varphi_{N_1,N_2}: \S_{N_1,N_2}\to\C$.
In order to determine the irreducibility of the $\hg$-module $\hM(\vphi_{N_1,N_2})$
and the $\tg$-module $\tM(\vphi_{N_1,N_2})$, we first reduce the problem to a special case using the change of the basis of the root system  (passing to an equivalent module).

For any $\hg$-module $M$ and an automorphism
$\si$ of $\hg$, we define a twisted $\hg$-module structure on $M$, denoted by $M^\si$, as follows: $x\cdot v=\si(x)v$ for any $x\in\hg$ and $v\in M$, where $x\cdot v$ means the module action in $M^\si$ and $\si(x)v$ means the module action in $M$.

For any $k\in\Z$, 
we have the following Lie algebra automorphism $\si_k$ of $\hg$ defined by
$$\aligned
&\si_k(e(n))=e(n+k),\quad \si_k(f(n))=f(n-k),\\
&\si_k(h(n))=h(n)+\delta_{n,0}kK,\quad \si_k(K)=K,
\endaligned$$
for any $n\in\Z$. It is clear that $\si_k=\si_1^k$.

Now it is easy to check that $\hM(\vphi_{N_1,N_2})^{\si_k}\cong \hM(\vphi_{N_1-k,N_2+k})$, where
$\vphi_{N_1-k,N_2+k}=\vphi_{N_1,N_2}\si_k|_{S_{N_1-k,N_2+k}}$ is a Lie algebra homomorphism from $\S_{N_1-k,N_2+k}$ to $\C$.
Note that $\vphi_{N_1-k,N_2+k}(h(N_1+N_2+1))=\vphi_{N_1,N_2}(h(N_1+N_2+1))$ and $\vphi_{N_1-k,N_2+k}(K)=\vphi_{N_1,N_2}(K)$. In particular, we have
\begin{equation}\label{reduce}
	\hM(\vphi_{N_1,N_2})^{\si_{N_1+1}}\cong \hM(\vphi_{-1,N_1+N_2+1}),
\end{equation}
and the irreducibility
of $\hM(\vphi_{N_1,N_2})$ is equivalent to the irreducibility of $\hM(\vphi_{-1,N_1+N_2+1})$. Hence, it is sufficient to consider only modules $\hM(\vphi_{-1,N})$.

\medskip

\subsection{Singular vectors in $\hM(\vphi_{-1,N})$}

 In this subsection we  consider the case $N_1=-1$ and  $N_2=N\geq1$.
Fix a  homomorphism $\varphi_N: \S_{-1,N}\to\C$ and denote $\vphi=\varphi_N$ and $\hM=\hM(\vphi_{N})$ for short.
As we agreed above, we denote by $u$ the generator of $\hM$ and  identify it with $1\ot u$.
In particular, we have $h(n)u=\vphi(h(n))u$ for all $n\in\Z_+$, $e(n)u=0$ for all $n\geq0$ and $f(n)u=h(n)u=0$ for all $n>N$.

For any $m\in\Z$, let $\Gamma_m$ be the set of tuples
$\gamma=(m_1,\dots,m_k), k\in\Z_+$, satisfying the following conditions
$$m_i\in\Z, m_i\leq m, 1\leq i\leq k\ \text{and}\ m_i\leq m_{i+1}, 1\leq i\leq k-1.$$
Here we  use the convention that $\gamma=\bo$ (has no entries) for $k=0$ above.

For $m\in\Z$ and $\gamma\in\Gamma_m$ as in the previous paragraph, we define
the \emph{height} of $\gamma$ by $\htt(\gamma)=\sum_{i=1}^km_i$,
and the \emph{length} of $\gamma$ by $\ell(\gamma)=k$.
For $\gamma=\bo$, we have $\htt(\gamma)=\ell(\gamma)=0$.
We also define the {\bf true height} of $\ga\in\Ga_m$ as $\Ht(\gamma)=\ell(\gamma)N-\htt(\gamma)$.

For any $\gamma=(m_1,\dots,m_k)\in\Gamma_m, m\in\Z$ and $x\in\g$, set
$$x(\gamma)=(x\ot t^{m_1})(x\ot t^{m_2})\cdots(x\ot t^{m_k})\in U(\hg)$$
and $x(\bo)=1$. In particular, in the case $\ell(\gamma)=1$, we have $x(\gamma)=x(m_1)$.
Note the difference between $x(\bo)$ and $x(0)$.

Let $\Gamma$ be the set of all triples $(\ga_+,\ga_0,\ga_-)$ with
$\ga_+,\ga_0\in\Gamma_{-1}$ and $\ga_-\in\Gamma_N$.
For any $\ga\in\Ga$, we write $\ga_+, \ga_0$ and $\ga_-$ for the unique elements in
$\Ga_{-1},\Ga_{-1}$ and $\Ga_{N}$ respectively such that $\ga=(\ga_+,\ga_0,\ga_-)$.
Denote $\htt(\ga)=\htt(\ga_+)+\htt(\ga_0)+\htt(\ga_-)$,
We have similar meanings for $\ell(\ga)$ and $\Ht(\ga)$.
It is clear that $\Ht(\ga)=\ell(\ga)N-\htt(\ga)\geq0$ for all $\ga\in\Ga$.

For any $\ga=(\ga_+,\ga_0,\ga_-)\in\Gamma$, we set $X(\ga)=e(\ga_+)h(\ga_0)f(\ga_-)$.
For convenience, we also write $\bo=(\bo,\bo,\bo)\in\Ga$ and set $X(\bo)=1$.
It is obvious that $X(\ga)u, \ga\in\Ga$, form a basis of $\hM$.

For any $w=\sum_{\ga\in\Ga}a_\ga X(\ga)u\in \hM$ with only finitely many nonzero coefficients $a_\ga\in\C$, $\ga\in\Ga$,
 we set
$$\Ht(w)=\max\{\Ht(\ga)\ |\ \ga\in\Ga, a_\ga\neq0\},$$
$$ \ell(w)=\max\{\ell(\ga)\ |\ \ga\in\Ga, a_\ga\neq0\},$$
called the true height and length of $w$ respectively.


For any $H, l\in\Z_+$, let us denote
$$J(H,l)=\{\ga'\in\Ga\ |\ \Ht(\ga')<H\ \text{or},\ \Ht(\ga')=H\ \text{and}\ \ell(\ga')\geq l\},$$
$$J'(H,l)=\{\ga'\in\Ga\ |\ \Ht(\ga')<H\ \text{or},\ \Ht(\ga')=H\ \text{and}\ \ell(\ga')> l\},$$
and
$$\bar J(H,l)=\{\ga'\in\Ga\ |\ \Ht(\ga')=H\ \text{and}\ \ell(\ga')=l\}.$$
For convenience, we also set
$$\hM(H,l)=\sum_{\ga\in J(H,l)}\C X(\ga)u\ \text{and}\ \hM'(H,l)=\sum_{\ga\in J'(H,l)}\C X(\ga)u.$$

Recall that a nonzero element $w\in \hM$ is a singular vector of weight $\lambda\in\C$ if
$$e(i-1)w=f(N+i)=(h(i)-\vphi(h(i)))=0$$
for any $i\in\N$ and $h(0)w=\l w$.

\begin{lem}\label{singular submod}
Let $w\in \hM$ be a nonzero singular vector. Then $U( \widehat{\sl}_2)w=\hM$ if and only if $w\in\C u$.
\end{lem}

\begin{proof}
Let $w\in \hM$ be a nonzero singular vector. If $w\in\C u$, it is clear that $U( \widehat{\sl}_2)w=\hM$ by definition.

Now suppose $w\notin\C u$. Then we have $\ell(w)\geq1$. From the definition of singular vectors and the PBW Theorem we get $U( \widehat{\sl}_2)w=\sum_{\ga\in\Ga}\C X(\ga)w$.  Using the PBW Theorem again we can  check that
any element in $U( \widehat{\sl}_2)w$ has length no less than $\ell(w)$. In particular, $u\notin U( \widehat{\sl}_2)w$ and hence $U( \widehat{\sl}_2)w\neq \hM$, as desired.
\end{proof}

\begin{exa}\label{phi(hN)=0}
Suppose $\vphi(h(N))=0$.  For each $m\in\Z_+$, we claim that  $(f(N))^mu$ is a singular vector of weight $\vphi(h(0))-2m$.
Indeed, let $w$ be a singular vector of weight $\l\in\C$, then we have
$$\aligned
&e(i)f(N)w=[e(i),f(N)]w=\vphi(h(N+i))w=0,\ \forall\ i\in\Z_+,\\
&(h(i)-\vphi(h(i)))f(N)w=[h(i),f(N)]w=-2f(N+i))w=0,\ \forall\ i\in\N,\\
&f(i)f(N)w=f(N)f(i)w=0,\ \forall\ i\in\Z, i\geq N+1,\\
&h(0)f(N)w=f(N)(h(0)-2)w=(\vphi(h(0))-2)f(N)w,
\endaligned$$
and $f(N)w$ is also a singular vector. Using induction on $m$, we deduce that $(f(N))^mu$ is a singular vector for any $m\in\Z_+$.
By Lemma \ref{singular submod}, we see that each element $(f(N))^mu, m\in\N$, generates
a proper nonzero submodule of $\hM$ and hence $\hM$ is not irreducible over $ \widehat{\sl}_2$ if $\vphi(h(N))=0$.
The main result of this section is to show that the converse holds, i.e., $\hM$ is irreducible if $\vphi(h(N))\neq0$ and the level is non-critical.
\end{exa}

\begin{exa}\label{phi(K)=-2}
Suppose $\vphi(K)=-2$. By \eqref{T(n)'}, the action of $T(n), n\in\Z$, is well defined on $\hM$.
Moreover, all $T(n)$ commute with the elements of $U(\hg)$ and, in particular, $[T(m),T(n)]=0$ for all $m,n\in \Z$.
It is easy to see that all elements in $\C[T(n),n\in\Z]u$ are singular vectors.

Moreover, we can check that $T(n)u\in\C u$ if $n\geq N$ and
$$T(n)u\equiv \frac{1}{2}h(n-N)h(N)u\mod \hM'(2N-n,1),\quad\text{if}\ n<N.$$
Hence we have $$\C[T(n),n\in\Z]u=\C[T(N-n), n\in\N]u.$$
\end{exa}

\begin{lem}\label{irre singular}
The $ \widehat{\sl}_2$-module $\hM$ is irreducible if and only if all its
singular vectors lie in $\C u$.
\end{lem}

\def\M{\widehat{M}}

\begin{proof} The necessity follows from Lemma \ref{singular submod}.
We only need to proof the sufficiency. Suppose that all singular vectors of
$\M$ lie in $\C u$.

Let $W$ be a nonzero submodule of $\M$ and let $w\in W$ be a nonzero element with minimal $\Ht(w)$.
Since $\M$ is a weight module, so is $W$, and we may assume that $w$ is a weight vector with  maximal weight
among all elements in $W$ with minimal $\Ht(w)$.
This is possible since only finitely many weights may occur for elements with the same true height.
In particular, $h(0)w=\l w$ for some $\l\in\C$.

Write $w=\sum_{\ga\in\Ga}a_{\ga}X(\ga)u$, where only finitely many $a_\ga\in\C$, $\ga\in\Ga$ are nonzero.
Then we have
\begin{equation}\label{w_ij}\aligned
w_{1,i}=&e(i)w=\sum_{\ga\in\Ga}a_{\ga}[e(i),X(\ga)]u\in W,\ \forall\ i\in\N,\\
w_{2,i}=&f(N+i)w=\sum_{\ga\in\Ga}a_{\ga}[f(N+i),X(\ga)]u\in W,\ \forall\ i\in\N,\\
w_{3,i}=&(h(i)-\vphi(h(i)))w=\sum_{\ga\in\Ga}a_{\ga}[h(i),X(\ga)]u\in W,\ \forall\ i\in\N.\\
\endaligned\end{equation}

Let $x(j)$ be one of $e(i), f(N+i), h(i), i\in\N
$. Then we have
\begin{equation}\label{x(j)w}\aligned
&\hskip-.2cm[x(j),X(\ga)]u=\sum_{\htt(\ga')=\htt(\ga)+j,\ \ell(\ga')\leq\ell(\ga)}b_{\ga'}X(\ga')u\\
&\hskip2cm+\sum_{n=0}^N\sum_{\htt(\ga'')+n=\htt(\ga)+j,\ \ell(\ga'')\leq\ell(\ga)-1}c_{\ga'',n}X(\ga'')h(n)u,
\endaligned\end{equation}
where only finitely many $b_{\ga'}\in\C$ or $c_{\ga',n}\in\C$ are nonzero.

For an element in the first summand in \eqref{x(j)w}, we have
$$\Ht(\ga')=\ell(\ga')N-\htt(\ga')\leq\ell(\ga)N-\htt(\ga)-j<\Ht(\ga).$$
For an element in the second summand in \eqref{x(j)w}, we have
$$
\Ht(\ga'')=\ell(\ga'')N-\htt(\ga'')\leq(\ell(\ga)-1)N-\htt(\ga)-j+n<\Ht(\ga).
$$
For any $j\in\N$, we see that $\Ht(x(j)w)<\Ht(w)$ provided $\Ht(x(j)w)\neq0$.
By the choice of $w$, we conclude that all  elements in \eqref{w_ij} are zero.

Similarly, by taking $x(j)=e(0)$ in the above argument, we can deduce that $\Ht(e(0)w)<\Ht(w)$, or $\Ht(e(0)w)=\Ht(w)$ and the weight of $e(0)w$ is strictly larger than the weight of $w$. From the choice of $w$ we must have
  $e(0)w=0$.

Combining the above results, we see that $w$ is a singular vector and it lies in $\C u$ by the assumption.
Thus $W=\M$, which must be irreducible.
\end{proof}

By the proof of the above lemma, we have immediately

\begin{cor}\label{submod singular}
Any nonzero $\hg$-submodule of $\hM$ has a nonzero singular vector.
\end{cor}

{\bf In the rest of this section, we always assume that $\vphi(h(N))\neq0$.}
Next we will study singular vectors in $\M$ which are not    multiples of $u$.
 Lt $w\in \M$ be a singular vector  which is not a   multiple of $u$.


Clearly,  there exist $H\in\Z_+$ and $l\in\N$ such that
\begin{equation}\label{w} w=\sum_{\ga\in J(H,l)}a_\ga X(\ga)u\end{equation} with only finitely many $a_\ga\in\C$ nonzero, and at least one $a_\ga\neq0$ for some $\ga\in\bar J(H,l)$.

\begin{lem}\label{H,l}
If $a_\ga\neq0$ for  $\ga\in\bar J(H,l)$, then  $\gamma_+=\bo\in\Ga_{-1}$ and
$\gamma_-=\bo\in\Ga_N$.
In other words, $X(\gamma)=h(\gamma_0)$ for some $\gamma_0\in\Gamma_{-1}$.
\end{lem}

\begin{proof} Fix  $\ga\in\bar J(H,l)$ with $a_\ga\neq0$.
 Suppose on the contrary that $\ga_-\neq\bo$ or $\ga_+\neq\bo$.

First assume that $\ga_-=(m_1,\dots,m_k)\neq\bo$ for some $m_1,\dots, m_k\in\Z$, $k\in\N$ such that
$m_1\leq m_2\dots\leq m_k\leq N$.
Set $\bar\ga_-=(m_1,\dots,m_{k-1})\in\Ga_N$ and $\bar\ga=(\ga_+,\ga_0,\bar\ga_-)\in\Ga$.
It is clear that $\Ht(\bar\ga)=\Ht(\ga)-N+m_k$ and $\ell(\bar\ga)=\ell(\ga)-1$.


Applying $e(N-m_k)$ to $w$, we have
\begin{equation}\label{e(N-m_k)}\aligned
e(N-m_k)w=&\sum_{\ga'\in J(H,l)}a_{\ga'}e(N-m_k)X(\ga')u\\
=&\sum_{\ga'\in J(H,l)}a_{\ga'}[e(N-m_k),X(\ga')]u.\\
\endaligned\end{equation}
Take any $\ga'=(\ga'_+,\ga'_0,\ga'_-)\in J(H,l)$. By the PBW Theorem, we can write
\begin{equation}\label{sum}\aligned
&[e(N-m_k),X(\ga')]\\
=&\sum_{\ga''\in I_1}b_{\ga''}X(\ga'')+\sum_{(\ga'',n)\in I_2}c_{\ga'',n}X(\ga'')\big(h(n)+\delta_{n,0}d_{\gamma''}K\big)+Y,
\endaligned\end{equation}
where $I_1$ consists of elements $\ga''\in\Ga$ with $\htt(\ga'')=\htt(\ga')+N-m_k$ and $\ell(\ga'')\leq \ell(\ga')$; $I_2$ consists of elements $(\ga'',n)\in\Ga\times\Z$ with
$0\leq n\leq N$, $\htt(\ga'')+n=\htt(\ga')+N-m_k$ and $\ell(\ga'')\leq \ell(\ga')-1$; and $Y\in U(\widehat{\sl}_2)(e\otimes\C[t])+U(\widehat{\sl}_2)((\C h\oplus\C f)\otimes t^{N+1})$. Note that only finitely many $b_{\ga''}\in\C, \ga''\in I_1$, or $c_{\ga'',n}\in\C, (\ga'',n)\in I_2$, are nonzero.
Noticing that $Yu=0$, we see that the summand $Y$ does not contribute to the element $e(N-m_k)w$ and we will
ignore it in the following discussion.

Consider the first summand in \eqref{sum} and take $\ga''\in I_1$.
We set $p=\ell(\ga')-\ell(\ga'')\in\Z_+$. Then we have
$$\aligned
\Ht(\ga'')=&\ell(\ga'')N-\htt(\ga'')=(\ell(\ga')-p)N-\htt(\ga')-N+m_k\\
=&\Ht(\ga')-N+m_k-pN\\
\leq&\Ht(\ga)-N+m_k-pN\leq\Ht(\bar\ga).\\
\endaligned$$
Thus $\Ht(\ga'')=\Ht(\bar\ga)$  if and only if $\Ht(\ga')=\Ht(\ga)$ and $p=0$.
We see that $\ga''=\bar\ga$ implies $\Ht(\ga')=\Ht(\ga)$ and $\ell(\ga')=\ell(\ga'')=\ell(\bar\ga)=\ell(\ga)-1$ and hence $\ga'\notin J(H,l)$. So $\ga''=\bar\ga$ never happens for any $\ga\in\bar J(H,l)$ with $a_\ga\neq0$ in this case.

Next we consider the second summand in \eqref{sum} and take $(\ga'',n)\in I_2$.
We set $p=\ell(\ga')-\ell(\ga'')\in\N$. Then
$$\aligned
\Ht(\ga'')=&\ell(\ga'')N-\htt(\ga'')=(\ell(\ga')-p)N-\htt(\ga')-N+m_k+n\\
=&\Ht(\ga')-N+m_k-pN+n\\
\leq&\Ht(\ga)-N+m_k-pN+N=\Ht(\bar\ga)-(p-1)N\leq\Ht(\bar\ga).\\
\endaligned$$
Hence $\Ht(\ga'')=\Ht(\bar\ga)$ if and only if $\Ht(\ga')=\Ht(\ga), n=N$ and $p=1$.
We see that $\ga''=\bar\ga$ implies $\ell(\ga')=\ell(\ga'')+1=\ell(\bar\ga)+1=\ell(\ga)$
and  $\ga'=\ga$. Indeed, suppose $\ga'=(\ga'_+,\ga'_0,\ga'_-)$ with
$\ga'_-=(m'_1,\dots,m'_{k'})\neq\bo$. Then $[e(N-m_k),X(\ga')]$ gives rise to $X(\bar\ga)h(N)$
(i.e., $c_{\bar\ga,N}\neq0$ in the second summand of \eqref{sum}) only if
$m'_i=m_k$ for some $1\leq i\leq k'$, $(m'_1,\dots,\widehat{m'_i},\dots,m'_{k'})=\bar\ga=(m_1,\dots,m_{k-1})$
and $\ga'_0=\ga_0$, $\ga'_+=\ga_+$. This implies $k=k'$ and $\ga'=\ga$ by the assumptions
$m_1\leq\dots\leq m_k$ and $m'_1\leq\dots\leq m'_k$.
In other words, $c_{\bar\ga,n}\neq0$ in the second summand of \eqref{sum} only if $\ga'=\ga$ and $n=N$.

From the above arguments we can deduce that
$$\aligned
&e(N-m_k)w\equiv a_\ga[e(N-m_k),X(\ga)]u\\
\equiv&\sum_{m_i=m_k} e(\ga_+)h(\ga_0)f(m_1)\dots f(m_{i-1})[e(N-m_k),f(m_i)]f(m_{i+1})\dots f(m_k)u,\\
\equiv&\sum_{m_i=m_k}e(\ga_+)h(\ga_0)f(m_1)\dots \widehat{f(m_i)}\dots f(m_k)h(N)u,\\
\equiv&|\{1\leq i\leq k\ |\ m_i=m_k\}|\varphi(h(N))X(\bar\ga)u\mod \sum_{\ga''\in\Ga\setminus\{\bar\ga\}}X(\ga'')u.\\
\endaligned$$
In particular, $e(N-m_k)w$ is nonzero, contradicting to the fact that $w$ is a singular vector.

If $\ga_+\neq\bo$, we can also deduce a contradiction by the similar arguments,
only applying $f(N-m_k)$ to $w$ if $\ga_+=(m_1,\dots,m_k)$ for some $m_i\in\Z, m_i<0$ for
$i=1,\dots,k\in\N$. The lemma follows.
\end{proof}

\begin{lem}\label{H,l+1}
Suppose $a_\ga\neq0$ for some $\ga\in\bar J(H,l)$ with $\ga_+=\bo\in\Ga_{-1}$, $\ga_-=\bo\in\Ga_N$
and $\ga_0=(m_1,\dots,m_l)\in\Ga_{-1}$. Fix any $1\leq k\leq l$ and set $\bar\ga_0=(m_1,\dots,\widehat{m_k}\dots,m_{l})$.
For any integer $1\leq i\leq -m_k$, define $\ga^{(i)}\in\Ga$ by
$$\ga^{(i)}_+=(-i)\in\Ga_{-1},\quad \ga^{(i)}_0=\bar\ga_0\in\Ga_{-1},\quad\ga^{(i)}_-=(N+m_k+i)\in\Ga_N. $$
Then we have $a_{\ga^{(i)}}=\frac{2a_\ga}{\varphi(h(N))}|\{1\leq j\leq l\ |\ m_j=m_k\}|$ for all $1\leq i\leq -m_k$. In particular, all $a_{\ga^{(i)}}, 1\leq i\leq -m_k$, are equal.
\end{lem}

\begin{proof} Fix any integer $1\leq i\leq -m_k$.
Applying $f(N+i)$ to $w$, we have
$$\aligned
f(N+i)w=&\sum_{\ga'\in J(H,l)}a_{\ga'}f(N+i)X(\ga')u\\
=&\sum_{\ga'\in J(H,l)}a_{\ga'}[f(N+i),X(\ga')]u\\
\endaligned$$
Set $\tilde\ga\in\Ga$ with $\tilde\ga_+=\bo\in\Ga_{-1}, \tilde\ga_0=\bar\ga_0, \tilde\ga_-=(N+i+m_k)\in\Ga_N$. Note that $\Ht(\tilde\ga)=\Ht(\bar\ga_0)-m_k-i=\Ht(\ga)-N-i, \ell(\tilde\ga)=\ell(\ga)=l$.

Take any $\ga'\in J(H,l)$. By the PBW Theorem, we can write
\begin{equation}\label{sum'}\aligned
&[f(N+i),X(\ga')]\\
=&\sum_{\ga''\in I_1}b_{\ga''}X(\ga'')+\sum_{(\ga'',n)\in I_2}c_{\ga'',n}X(\ga'')\big(h(n)+\delta_{n,0}d_{\gamma''}K\big)+Y,
\endaligned\end{equation}
where $I_1$ consists of $\ga''\in\Ga$ with $\htt(\ga'')=\htt(\ga')+N+i$ and $\ell(\ga'')\leq \ell(\ga')$; $I_2$ consists of $(\ga'',n)\in\Ga\times\Z$ with
$0\leq n\leq N$, $\htt(\ga'')+n=\htt(\ga')+N+i$ and $\ell(\ga'')\leq \ell(\ga')-1$; $Y\in U(\widehat{\sl}_2)(e\otimes\C[t])+U(\widehat{\sl}_2)((\C h\oplus\C f)\otimes t^{N+1})$ and only finitely many $b_{\ga''}\in\C, \ga''\in I_1$, or $c_{\ga'',n}\in\C, (\ga'',n)\in I_2$, are nonzero. Note that $Yu=0$.

Consider the first summand in \eqref{sum'} and take $\ga''\in I_1$.
Set $p=\ell(\ga')-\ell(\ga'')\in\Z_+$. We have
$$\aligned
\Ht(\ga'')=&\ell(\ga'')N-\htt(\ga'')=(\ell(\ga')-p)N-\htt(\ga')-N-i\\
=&\Ht(\ga')-N-i-pN\\
\leq&\Ht(\ga)-N-i-pN=\Ht(\tilde\ga)-pN\leq\Ht(\tilde\ga).\\
\endaligned$$
Then $\Ht(\ga'')=\Ht(\tilde\ga)$ if and only if $\Ht(\ga')=\Ht(\ga)$ and $p=0$.
Suppose $\ga''=\tilde\ga$. We get $\Ht(\ga')=\Ht(\ga)$ and $\ell(\ga')=\ell(\ga'')=\ell(\tilde\ga)=\ell(\ga)$ and hence $\ga'\in \bar J(H,l)$. By Lemma \ref{H,l}, we see that $\ga'_+=\bo\in\Ga_{-1}$, $\ga'_-=\bo\in\Ga_{N}$ and $\ga'_0=(m'_1,\dots,m'_{k'})\in\Ga_{-1}$.
Then 
$[f(N+i),X(\ga')]$ gives rise to $\ga''=\tilde\ga$
only if $m'_j=m_k$ for some $1\leq j\leq k'$
and $(m'_1,\dots,\widehat{m'_{j}},\dots,m'_{k'})=(m_1,\dots,\widehat{m_k},\dots,m_{l})=\bar\ga_0$, forcing $k'=l$ and $\ga'=\ga$. So $b_{\tilde\ga}\neq0$ in the first summand of \eqref{sum'} only if $\ga'=\ga$.

Now we consider the second summand in \eqref{sum'} and take $(\ga'',n)\in I_2$.
Again set $p=\ell(\ga')-\ell(\ga'')\in\N$. We have
$$\aligned
\Ht(\ga'')=&\ell(\ga'')N-\htt(\ga'')=(\ell(\ga')-p)N-\htt(\ga')-N-i+n\\
=&\Ht(\ga')-i-(p+1)N+n\\
\leq&\Ht(\ga)-i-pN=\Ht(\tilde\ga)-(p-1)N\leq\Ht(\tilde\ga).\\
\endaligned$$
Then $\Ht(\ga'')=\Ht(\tilde\ga)$ if and only if $\Ht(\ga')=\Ht(\ga), n=N$ and $p=1$.
Suppose $\ga''=\tilde\ga$. Then $\ell(\ga')=\ell(\ga'')+1=\ell(\tilde\ga)+1=\ell(\ga)+1$.
Write $\ga'=(\ga'_+,\ga'_0,\ga'_-)$. Hence $[f(N+i),X(\ga')]$ gives rise to $X(\tilde\ga)h(N)$ if and only if
$\ga'_+=(-i)\in\Ga_{-1}$, $\ga'_0=\bar\ga_0\in\Ga_{-1}$ and $\ga'_-=(N+i+m_k)\in\Ga_N$.
That is, $c_{\tilde\ga,n}\neq0$ in the second summand of \eqref{sum'} only if $\ga'=\ga^{(i)}$ and $n=N$.

From the above arguments we  deduce that
$$\aligned
&f(N+i)w\equiv a_\ga[f(N+i),X(\ga)]u+a_{\ga^{(i)}}[f(N+i),X(\ga^{(i)})]u\\
\equiv&a_\ga\sum_{m_j=m_k}h(m_1)\cdots h(m_{j-1})[f(N+i),h(m_j)]h(m_{j+1})\cdots h(m_l)u\\
      &+a_{\ga^{(i)}}[f(N+i),e(-i)]h(m_1)\cdots\widehat{h(m_k)}\cdots h(m_{l})f(N+i+m_k)u\\
\equiv&2a_\ga\sum_{m_j=m_k}h(m_1)\cdots \widehat{h(m_{j})}\cdots h(m_l)f(N+i+m_k)u\\
      &-a_{\ga^{(i)}}h(N)h(m_1)\cdots\widehat{h(m_k)}\cdots h(m_{l})f(N+i+m_k)u\\
\equiv&2a_\ga|\{1\leq j\leq l\ |\ m_j=m_k\}|h(m_1)\cdots\widehat{m_k}\cdots h(m_{l})f(N+i+m_k)u\\
      &-a_{\ga^{(i)}}h(m_1)\cdots\widehat{h(m_k)}\cdots h(m_{l})f(N+i+m_k)h(N)u\\
\equiv&\Big(2a_\ga|\{1\leq j\leq l\ |\ m_j=m_k\}|-a_{\ga^{(i)}}\varphi(h(N))\Big)\\
&\hskip0.1cm\cdot h(m_1)\cdots\widehat{h(m_k)}\cdots h(m_{l})f(N+i+m_k)u\mod\sum_{\ga''\in\Ga\setminus\{\tilde\ga\}}X(\ga'')u.\\
\endaligned$$
Since $w$ is a singular vector, we have $f(N+i)w=0$. The lemma now follows.
\end{proof}

\subsection{Irreducibility in the noncritical case }

Now we consider the case $\vphi(K)\ne -2$. We use notations and assumptions as before, in particular, $\vphi(h(N))\neq0$.

\begin{lem}\label{w in C}
If $\vphi(K)\neq-2$, then any singular vector  in $\hM$ lies in $\C u$.
\end{lem}

\begin{proof}
 To the contrary, suppose that $\hM_{-1,N}(\vphi)$ has  a singular vector   $w\notin\C u$ as in (\ref{w}). Then there exist some $H,l\in\N$ such that
$w\in \M(H,l)$ and $a_\ga\neq0$ for some $\ga\in\bar J(H,l)$.
By Lemma \ref{H,l}, we have $\ga=(\ga_+,\ga_0,\ga_-)$ with $\ga_+=\bo\in\Ga_{-1}$, $\ga_-=\bo\in\Ga_N$ and
$\ga_0=(m_1,\dots,m_l)\in\Ga_{-1}$,
where $m_1,\dots,m_l\in\Z, l\in\N$ and $m_1\leq\dots\leq m_l\leq -1$.
Note that $\ell(\ga)=l\geq 1$. Fix any $1\leq k\leq l$.

Set $\ga^{(i)}=(\ga^{(i)}_+,\ga^{(i)}_0,\ga^{(i)}_-)$ for any integer $1\leq i\leq -m_k$, where
$$\aligned
&\ga^{(i)}_+=(-i)\in\Ga_{-1},\quad \ga^{(i)}_-=(N+m_k+i)\in\Ga_{N},\\
&\ga^{(i)}_0=\bar\ga_0=(m_1,\dots,\widehat{m_k},\dots,m_l)\in\Ga_{-1}.\
\endaligned$$
Then we have $a_{\ga^{(i)}}=\frac{2a_\ga}{\varphi(h(N))}|\{1\leq j\leq l\ |\ m_j=m_k\}|\neq0$ for all $1\leq i\leq -m_k$ by Lemma \ref{H,l+1}.
Set $\bar\ga=(\bar\ga_+,\bar\ga_0,\bar\ga_-)$, where $\bar\ga_+=\bo\in\Ga_{-1}$ and $\bar\ga_-=\bo\in\Ga_N$.
It is clear that $\Ht(\bar\ga)=\Ht(\ga)-N+m_k$ and $\ell(\bar\ga)=\ell(\ga)-1$.

Applying $h(-m_k)-\varphi(h(-m_k))$ to $w$ we get
$$\aligned
(h(-m_k)-\varphi(h(-m_k)))w=&\sum_{\ga'\in J(H,l)}a_{\ga'}(h(-m_k)-\varphi(h(-m_k)))X(\ga')u\\
=&\sum_{\ga'\in J(H,l)}a_{\ga'}[h(-m_k),X(\ga')]u.\\
\endaligned$$
Take any $\ga'=(\ga'_+,\ga'_0,\ga'_-)\in J(H,l)$.
If $\ga'\in \bar J(H,l)$, then $\ga'_+=\bo\in\Ga_{-1}$ and $\ga'_-=\bo\in\Ga_N$ by Lemma \ref{H,l}.
Hence $[h(-m_k),X(\ga')]$ give rise to $X(\bar\ga)$ if and only if $\ga'=\ga$. 
In this case, we have
\begin{equation}\label{vphi(K)}
[h(-m_k),X(\ga)]u=-2m_k\vphi(K)|\{1\leq j\leq l\ |\ m_j=m_k\}|X(\bar\ga).
\end{equation}

 In the rest of the proof we assume that $\ga'\in J'(H,l)$, that is, either $\Ht(\ga')<H$ or $\Ht(\ga')=H$ and $\ell(\ga')>l$.
Using the PBW Theorem we can write
\begin{equation}\label{sum''}\aligned
&[h(-m_k),X(\ga')]\\
=&\sum_{\ga''\in I_1}b_{\ga''}X(\ga'')+\sum_{(\ga'',n)\in I_2}c_{\ga'',n}X(\ga'')h(n)+Y,
\endaligned\end{equation}
where $I_1$ consists of $\ga''\in\Ga$ with $\htt(\ga'')=\htt(\ga')-m_k$ and $\ell(\ga'')\leq \ell(\ga')$; $I_2$ consists of $(\ga'',n)\in\Ga\times\Z$ with
$0\leq n\leq N$, $\htt(\ga'')+n=\htt(\ga')-m_k$ and $\ell(\ga'')\leq \ell(\ga')-2$; $Y\in U(\widehat{\sl}_2)(e\otimes\C[t])+U(\widehat{\sl}_2)((\C h\oplus\C f)\otimes t^{N+1})$ and only finitely many $b_{\ga''}\in\C, \ga''\in I_1$, or $c_{\ga'',n}\in\C, (\ga'',n)\in I_2$, are nonzero. Note again that $Yu=0$.

Consider the first summand in \eqref{sum''} and take $\ga''\in I_1$.
Again we set $p=\ell(\ga')-\ell(\ga'')\in\Z_+$ and obtain
$$\aligned
\Ht(\ga'')=&\ell(\ga'')N-\htt(\ga'')=(\ell(\ga')-p)N-\htt(\ga')+m_k\\
=&\Ht(\ga')+m_k-pN\\
\leq&\Ht(\ga)+m_k-pN=\Ht(\bar\ga)-(p-1)N.\\
\endaligned$$
Suppose first that $p\geq1$. Then we have $\Ht(\ga'')\leq\Ht(\bar\ga)$ and the equality holds if and only if $\Ht(\ga')=\Ht(\ga)$ and $p=1$.
We see that $\ga''=\bar\ga$ implies $\Ht(\ga')=\Ht(\ga)$, $\ell(\ga')=\ell(\ga'')+1=\ell(\bar\ga)+1=\ell(\ga)$ and hence $\ga'\in\bar J(H,l)$.
This contradicts the assumption that  $\ga'\in J'(H,l)$.

Now suppose that $p=0$ and $\ga''=\bar\ga$.  We have   $\ell(\ga')=\ell(\ga'')=\ell(\bar\ga)$. One can  easily see that
$[h(-m_k),X(\ga')]$ can not give rise to $X(\ga'')=X(\bar\ga)=h(\bar\ga_0)$.

Consider the second summand in \eqref{sum''} and take $(\ga'',n)\in I_2$.
Set $p=\ell(\ga')-\ell(\ga'')\ge2$. Then
$$\aligned
\Ht(\ga'')=&\ell(\ga'')N-\htt(\ga'')=(\ell(\ga')-p)N-\htt(\ga')+m_k+n\\
=&\Ht(\ga')+m_k-pN+n\\
\leq&\Ht(\ga)+m_k-pN+N=\Ht(\bar\ga)-(p-2)N.\\
\endaligned$$\begin{flushleft}
	
\end{flushleft}
 We have $\Ht(\ga'')\leq\Ht(\bar\ga)$ and the equality holds if and only if $\Ht(\ga')=\Ht(\ga), n=N$ and $p=2$. Thus, the equality $\ga''=\bar\ga$ implies $\ell(\ga')=\ell(\ga'')+2=\ell(\bar\ga)+2=\ell(\ga)+1$. We will show that
 $\ga'=\ga^{(i)}$ for some $1\leq i\leq -m_k$.
Indeed, if $\ga'=(\ga'_+,\ga'_0,\ga'_-)$ then $[h(-m_k),X(\ga')]$ gives rise to $X(\bar\ga)h(N)=h(\bar\ga_0)h(N)$
if and only if $\ga'_+=(i)\in\Ga_{-1}, \ga'_-=(j)\in\Ga_N$ and $\ga'_0=\bar\ga_0$ for some $i,j\in\Z$ with $i\leq-1, j\leq N$ such that $[[h(-m_k),e(i)],f(j)]\in\C h(N)$. This forces $i+j-m_k=N$, or $j=N+m_k-i$,
and hence $\ga'=\ga^{(-i)}$ with $m_k\leq i\leq -1$.


From the above arguments we  deduce that
$$\aligned
&(h(-m_k)-\vphi(h(-m_k)))w\\
\equiv &a_\ga [h(-m_k),X(\ga)]u+\sum_{i=1}^{-m_k}a_{\ga^{(i)}}[h(-m_k),X(\ga^{(i)})]u,\\
\endaligned$$
and
$$\aligned
      &[h(-m_k),e(-i)h(\bar\ga_0)f(N+m_k+i)]u\\
\equiv&[h(-m_k),e(-i)]h(\bar\ga_0)f(N+m_k+i)u\\
\equiv&2e(-m_k-i)h(\bar\ga_0)f(N+m_k+i)u\\
\equiv&2h(\bar\ga_0)[e(-m_k-i),f(N+m_k+i)]u\\
\equiv&2\varphi(h(N))X(\bar\ga)u\ \mod\sum_{\ga''\in\Ga\setminus\{\bar\ga\}}\alpha_{\gamma''}X(\ga'')u.\\
\endaligned$$
Combining the above formula with \eqref{vphi(K)}, we obtain that the coefficient of $X(\bar\ga)$ in $(h(-m_k)-\vphi(h(-m_k)))w$ equals
$$-2a_\ga m_k(\vphi(K)+2)|\{1\leq j\leq l\ |\ m_j=m_k\}|,$$
which is nonzero by our assumptions, contradicting to the fact that $w$ is a singular vector. So, we must have $w\in\C u$, as required.
%
%
%
\end{proof}

\begin{thm}\label{irre}
The $ \widehat{\sl}_2$-module $\hM(\vphi)$ is irreducible if and only if \\$\vphi(h(N))\neq0$
and $\vphi(K)\neq-2$.
\end{thm}

\begin{proof}
The necessity follows from Lemma \ref{irre singular}, Lemma \ref{singular submod}, Example \ref{phi(hN)=0}
and Example \ref{phi(K)=-2}.
The sufficiency follows from Lemma \ref{irre singular} and Lemma \ref{w in C}.
\end{proof}

By the isomorphism in \eqref{reduce}, we can obtain one of our main results.

\begin{thm}\label{irre general}
Suppose $N_1, N_2\in\Z$ with $N_1+N_2\geq0$ and $\vphi$ is a Lie algebra homomorphism from $\S_{N_1,N_2}$ to $\C$.
The $\hg$-module $\hM(\vphi)$ is irreducible if and only if $\vphi(h(N_1+N_2+1))\neq0$
and $\vphi(K)\neq-2$.
\end{thm}

\subsection{Irreducibility in the critical case}

In this subsection, we always assume that $N_1=-1, N_2=N\ge1$, $\vphi=\vphi_{-1,N}$, $\vphi(h(N))\neq0$ and $\vphi(K)=-2$.

For any $m\in\Z$ and $\ga=(m_1,\dots,m_l)\in\Z^m$ with $m_i\leq m$, denote $\vec{\ga}=(m_{\si(1)},\dots,m_{\si(l)})\in\Ga_m$, where $\si$ is a permutation on $\{1,\dots,l\}$ such that $m_{\si(1)}\leq m_{\si(2)}\dots \leq m_{\si(l)}$.
For any $\ga=(m_1,\dots,m_l)$ and $\ga'=(m'_1,\dots,m'_{l'})\in\Ga_m$, we define the ordered juxtaposition
$$\ga\Join\ga'=\overrightarrow
{(m_1,\dots,m_l,m'_1,\dots,m'_{l'})}\in\Ga_m.$$
And for $\ga, \ga'\in\Ga$, denote $\ga\Join\ga'=(\ga_+\Join\ga'_+,\ga_0\Join\ga'_0,\ga_-\Join\ga'_-)\in\Ga$.

\begin{lem}\label{Ht property}
For any $\ga,\ga'\in\Ga$, we have that $X(\ga)X(\ga')-X(\ga\Join\ga')$ is a sum of elements
of the form $X(\ga'')$ with $\Ht(\ga'')<\Ht(\ga)+\Ht(\ga')$.
\end{lem}

\begin{proof}
By induction on $\ell(\gamma)$ and $\ell(\gamma')$, we need only to prove for the case $\ell(\ga)=\ell(\ga')=1$ which  follows from the fact that $$\Ht(x(i)y(j)u)=\Ht(y(j)x(i)u)>\Ht([x,y](i+j)u)$$
for any $x,y\in\{e,h,f\}$ with all vectors in the formula nonzero.
\end{proof}


Now consider the operators defined in Subsection \ref{affineVA}.
Noticing that $f(i)u=h(i)u=e(i-N-1)u=0$ for all $i>N$, we have the following equalities
  for any $n<N$:
{\Small$$\aligned
4T(n)u=&\sum_{i\in\Z}\big(2:e(-i)f(n+i):+2:f(-i)e(n+i):\\
&+:h(-i)h(n+i):\big)u\\
=&\big(2\sum_{i=1}^{N-n} e(-i)f(n+i)+2\sum_{i=1}^\infty f(-i)e(n+i)+2\sum_{i=-N}^0 e(n+i)f(-i)\\
&+\sum_{i=1}^{N-n} h(-i)h(n+i)+\sum_{i=-N}^0 h(n+i)h(-i)\big)u=\tau(n)u,
\endaligned$$}
where
{\Small\begin{equation}\label{tau}\aligned
\tau(n) = &\Big(4\sum_{i=1}^{N-n} e(-i)f(n+i)+2(n+1)(h(n)+\delta_{n,0}i(e|f)K)\\
&+2\sum_{i=1}^{N-n} h(-i)h(n+i)+\sum_{j=-n}^0 h(-j)h(n+j)\Big) , {\text{ if }} 0\leq n<N,\\
 \tau(n)=&\Big(4\sum_{i=1}^{N-n} e(-i)f(n+i)+2(n+1)h(n)+\sum_{i=1}^{N-n} h(-i)h(n+i)\\
&+\sum_{j=-n}^{N-n} h(-j)h(n+j)\Big), {\text{ if }} n<0.
\endaligned\end{equation}}

For convenience, we denote
$$T'(n)=\frac{T(n)}{2\vphi(h(N))},\ \text{and}\ \tau'(n)=\frac{\tau(n)}{2\vphi(h(N))}.$$
Then we have
$(\tau'(n)-h(n-N))u\in M'(H,l)$ where $H=\Ht(h(n-N))=\Ht(\tau(n))=2N-n$ and $l=\ell(h(n-N))=1$ for any $n<N$.
Moreover, for any $n,n'<N$, we have
$$\aligned
&\tau(n)\tau(n')u=\tau(n)T(n')u=T(n')\tau(n)u=T(n')T(n)u\\
=&T(n)T(n')u=T(n)\tau(n')u=\tau(n')T(n)u=\tau(n')\tau(n)u.
\endaligned$$
Hence for any polynomial $P$ in $k$ variables, we have
$P(\tau(n_1),\dots,\tau(n_k))u=P(T(n_1),\dots,T(n_k))u$.

We define a total order on $\Z_+\times\Z_+$ as follows: For any $(H,l),(H',l')\in\Z_+\times\Z_+$,
write $(H,l)\succeq(H',l')$ if $H>H'$ or $H=H',l\leq l'$.

\begin{lem}\label{w in C[T]u}
Suppose $\vphi(K)=-2$. Then $\C[T(N-i), i\in\N]u$ is the set of all singular vectors in $\M(\vphi)$.
\end{lem}

%

\begin{proof} It is clear that nonzero vectors in $\C[T(N-i), i\in\N]u$ are singular vectors in $\hM=\M(\vphi)$.
	
Let $w\neq\C u$ be a singular vector in $\hM$ with $w\in\hM(H,l)$ and $w\notin\hM'(H,l)$.
By Lemma \ref{H,l}, we have
$$w\equiv P(h(-1),\dots,h(-m))u \mod \hM'(H,l),$$
where $P(h(-1),\dots,h(-m))$ is a polynomial in $h(-1),\dots,h(-m)$ for some $m\in\N$
such that each of its monomial is a scalar multiple of $h(\ga)$ for some
$\ga\in\Ga_{-1}$ with $\Ht(\ga)=H$ and $\ell(\ga)=l$. Then we have $H\geq (N+1)l$.

We continue by the induction on $(H,l)$ with respect to $\succeq$.
By the previous argument, we have
$$T'(N-i)u=\tau(N-i)u\equiv h(-i)u\mod \hM'(N+i,1),\ \forall\ i\in\N.$$
Then we get a new singular vector
$$\aligned
&P\big(T'(N-1),\dots,T'(N-n)\big)u-w\\
=&P\big(\tau'(N-1),\dots,\tau'(N-n)\big)u-w\in\hM'(H,l),
\endaligned$$
by Lemma \ref{Ht property}. This element, if nonzero, is a singular vector in $\hM'(H,l)$
and hence lies in $\C[T(N-i), i\in\N]u$ by the induction hypothesis.
Thus we obtain $w\in \C[T(N-i), i\in\N]u$ as desired.
\end{proof}

Let $ \Ga^{(0)}$ be the subset of $\Ga$ consisting of elements of the form $\ga=(\ga_+,\ga_0,\ga_-)\in\Ga$ with $\ga_0=\bo$.
For any $H, l\in\Z_+$, we denote $J(H,l)_0=J(H,l)\cap \Ga^{(0)}$, $\bar J(H,l)_0=\bar J(H,l)\cap \Ga^{(0)}$ and $J'(H,l)_0=J'(H,l)\cap \Ga^{(0)}$.
It is also convenient to denote
$$J'(H)_0=\{\ga\in \Ga^{(0)}\ |\ \Ht(\ga)< H\} $$
for any $H\in\Z_+$.

\begin{lem}\label{h to tau}
For any $\ga\in\Ga$ where $\ga_0=(m_1,\dots,m_k)\neq\bo$ with all negative $m_i\in\Z$, we have
$$\aligned
X(\ga)u=& X(\ga^{(0)})(h(m_1)-\tau'(N+m_1))\cdots(h(m_k)-\tau'(N+m_k))u+w',
\endaligned$$
for some $w'\in\sum_{\ga'\in J'(H)_0}X(\ga')\C[T(N-i), i\in\N]u$
and
$$X(\ga)u\in \sum_{\ga''\in J'(H,l+k-1)_0}X(\ga'')\C[T(N-i), i\in\N]u$$
where $\ga^{(0)}=(\ga_+,\bo,\ga_-)\in \Ga^{(0)}$, $H=\Ht(\ga)$ 
and $l=\ell(\ga)$.
\end{lem}

\begin{proof}
We prove the first assertion by the induction on $H$.
As before, set $\bar\ga_0=(m_1,\dots,m_{k-1})$ and $\bar\ga=(\ga_+,\bar\ga_0,\ga_-)$. Note that $$M'(H, l)=\sum_{\ga'\in J'(H, l)_0}X(\ga')\C[T(N-i), i\in\N]u$$
 is stable under the action of $h(m_k)$. Using the induction hypothesis repeatedly,
we obtain the following equivalences  modulo $M'(H, l)$:
$$\aligned
&X(\ga)u\equiv h(m_k)X(\bar\ga)u\\ 
\equiv& h(m_k)X(\ga^{(0)})(h(m_1)-\tau'(N+m_1))\cdots(h(m_{k-1})-\tau'(N+m_{k-1}))u\\
\equiv& X(\ga^{(0)})(h(m_1)-\tau'(N+m_1))\cdots(h(m_{k-1})-\tau'(N+m_{k-1})) h(m_k)u\\
\equiv& X(\ga^{(0)})(h(m_1)-\tau'(N+m_1))\cdots(h(m_k)-\tau'(N+m_k))u.\\
\endaligned$$

The result follows by the
induction on $H$, replacing $u$ by some elements in $\C[T(N-i), i\in\N]u$ if necessary.

Now, for the second assertion we note that
$$X(\ga^{(0)})(h(m_1)-\tau'(N+m_1))\cdots(h(m_k)-\tau'(N+m_k))u\in \widehat M'(H,l+k-1),$$
by \eqref{tau} and Lemma \ref{Ht property}. Then the statement follows
 by the induction on $(H, l)$. \end{proof}

For any $\theta_i\in\C, i\in\N$, denote by $\hM(\theta_i, i\in\N)$ the submodule of $\hM=\hM(\vphi)$ generated by the singular vectors
$$(T(N-i)-\theta_i)u=(\tau(N-i)-\theta_i)u, i\in\N,$$
and consider the quotient module
$$\bM=\bM(\vphi;\theta_i, i\in\N):=\hM(\vphi)/\hM(\theta_i,i\in\N).$$

For any $w\in\hM$, we use $\bar w$ to denote its image in $\bM$.
From Lemma \ref{h to tau} we see that $X(\ga)\bu, \ga\in \Ga^{(0)}$ form a basis of $\bM$.
For any $H, l\in\Z_+$, denote
$$\bM(H,l)=\sum_{\ga\in J(H,l)_0}\C X(\ga)\bu\ \text{and}\ \bM'(H,l)=\sum_{\ga\in J'(H,l)_0}\C X(\ga)\bu.$$
It is also convenient to denote $\bM'(H)=\sum_{\ga\in J'(H)_0}\C X(\ga)\bu$.

A nonzero element $w\in\bM$ is called a {\bf singular vector} of weight $\lambda\in\C$ if
$$e(N_1+i)w=f(N_2+i)=(h(i)-\vphi(h(i)))=0$$ for any $i\in\N$ and $hw=\l w$.
 The following analogue of Lemma \ref{irre singular} can be shown in a similar way, so we omit the proof.

\begin{lem}\label{irre singular bM}
The $\widehat{\sl}_2$-module $\bM(\vphi;\theta_i,  i\in\N)$ is irreducible if and only if all its
singular vectors lie in $\C\bu$.
\end{lem}

Now we can determine the irreducibility of $\bM$ and describe the maximal submodules of $\hM$ as a consequence.

\begin{thm}\label{irre K=-2}
Suppose $\vphi(h(N))\neq0$ and
$\vphi(K)=-2$. Then the $\hsl_2$-module $\bM(\vphi;\theta_i,  i\in\N)$ is irreducible for any
$\theta_i\in\C$.
\end{thm}


\begin{proof} We will show that all singular vectors of $\bM$ lie in $\C\bu$.
To the contrary, take a singular vector $w\in\bM$ which is not a multiple of $\bu$.
There exist $H\in\Z_+$ and $l\in\N$ such that
$$w=\sum_{\ga'\in J(H,l)_0}a_{\ga'} X(\ga')\bu$$
with finitely many $a_{\ga'}\in\C$ nonzero and at least one $a_{\ga'}\neq0$ for some $\ga'\in\bar J(H,l)_0$.

Fix some $\ga\in\bar J(H,l)_0$ with $a_\ga\neq0$. Then we have $\ga_-\neq\bo$ or $\ga_+\neq\bo$.

First assume $\ga_-=(m_1,\dots,m_k)\neq\bo$ for some $m_1,\dots,m_k\in\Z$ and $k\in\N$. 
Set $\bar\ga_-=(m_1,\dots,m_{k-1})\in\Ga_N$ and $\bar\ga=(\ga_+,\bo,\bar\ga_-)\in\Ga^{(0)}$.
It is clear that $\Ht(\bar\ga)=H-N+m_k$ and $\ell(\bar\ga)=l-1$.

Now we will show that $$e(N-m_k) w=e(N-m_k)\sum_{\ga'\in J(H,l)_0}a_{\ga'} X(\ga')\bu\ne0,$$ which contradicts the fact that $w$ is a singular vector.

Note that
\begin{equation}\label{3.15e(N-m_k)}\aligned
	  e(N-m_k)&\sum_{\ga'\in J(H,l)_0}a_{\ga'} X(\ga')\bu=\sum_{\ga'\in J(H,l)_0}a_{\ga'}e(\ga'_+)e(N-m_k)f(\ga'_-)\bar u\\
	=&\sum_{\ga'\in J(H,l)_0}a_{\ga'}e(\ga'_+)[e(N-m_k),f(\ga'_-)]\bar u.\\
	\endaligned\end{equation}
Take any $\ga'=(\ga'_+,\ga'_0,\ga'_-)\in J(H,l)_0$. By the PBW Theorem, we can write
\begin{equation}\label{3.15sum}\aligned
	&e(\ga'_+)[e(N-m_k),f(\ga'_-)]\\
	=&\sum_{\ga''\in I_1}b_{\ga''}X(\ga'')+\sum_{(\ga'',n)\in I_2}c_{\ga'',n}X(\ga'')\big(h(n)+\delta_{n,0}d_{\gamma''}K\big)+Y,
	\endaligned\end{equation}
where $I_1$ consists of elements $\ga''\in\Ga^{(0)}$ with $\ga''_+=\ga'_+, \ell(\ga''_-)=\ell(\ga'_-)-1$;
$I_2$ consists of elements $(\ga'',n)\in\Ga^{(0)}\times\Z$ with $\ga''_+=\ga'_+,\Ht(\ga'')\le H-N+m_k$ and
  $\ell(\ga'')=\ell(\ga')-1$;
  and $Y\in U(\widehat{\sl}_2)(e\otimes\C[t])+U(\widehat{\sl}_2)((\C h\oplus\C f)\otimes t^{N+1})$.
  Note
   that only finitely many $b_{\ga''}\in\C$ for $\ga''\in I_1$ and  finitely many $c_{\ga'',n}\in\C$ for $(\ga'',n)\in I_2$, are nonzero.
  Clearly, $Y\bar u=0$.

Consider the first summand in \eqref{3.15sum} and take $\gamma''\in I_1$. There exist $n_i$, $n_j\in\Z$ with $n_i,n_j\le N$ satisfying
$$\Ht(\gamma')-(2N-n_i-n_j)=\Ht(\gamma'')-(N-(n_i+n_j+N-m_k)).$$
Then $\Ht(\gamma'')=\Ht(\gamma')-2N+m_k\le H-2N+m_k$ and $\gamma''\neq \bar\gamma$ in this case.

Now take $(\ga'',n)\in I_2$ with $n\geq 0$. Then
$$\Ht(\gamma')-\Ht(\gamma'')=N-(n-(N-m_k))=2N-n-m_k.$$
Suppose $\gamma''=\bar\gamma$. Then we have $\Ht(\gamma')=H+N-n$. So $n=N$ and  $\gamma'=\gamma$.

For $(\ga'',n)\in I_2$ with $n<0$, we repeatedly use (\ref{tau}) to remove $h(n)$ and  write  $X(\gamma'')h(n)\bar u$ as a sum with terms of the form $X(\gamma''')\bar u$ with $\ga'''\in\Ga^{(0)}$.
Moreover,  \eqref{tau} and Lemma \ref{Ht property} imply
$$\Ht(\ga''')<\Ht(\ga'')+N-n=\Ht(\gamma')-N+m_k\leq H-N+m_k,$$
or $\Ht(\ga''')=\Ht(\gamma')-N+m_k$ and $\ell(\ga''')=\ell(\gamma'')+2=\ell(\gamma')+1$.
If $\ga'''=\bar\gamma$ then $\Ht(\gamma')=H$ and $l-1=\ell(\ga''')=\ell(\ga')+1\geq l+1$, which is a contradiction.

Thus,  $e(N-m_k)\bw\neq0$ in $\bM$, contradicting the fact that
$\bw$ is a singular vector.

If $\ga_+\neq\bo$, we get a contradiction by similar arguments,
 applying $f(N-m_k)$ to $\bw$ if $\ga_+=(m_1,\dots,m_k)$ for some $m_i\in\Z, m_i<0$,
$i=1,\dots, k\in\N$.  We conclude that all singular vectors of $\bM$ lie in $\C\bu$. Now the statement of the theorem follows from
 Lemma \ref{irre singular bM}.
\end{proof}

As a consequence, we give a description of all irreducible subquotients of $\hM=\hM(\vphi)$ in the critical case.
By Lemma \ref{h to tau}, we can write any nonzero $w\in  \widehat M$ in the form
\begin{equation}\label{w J_w}
w=\sum_{\gamma'\in J_w}X(\gamma')P_{\gamma'} u,
\end{equation}
where $J_w$ is a finite subset of $J(H,l)_0$ and $P_{\gamma'}\in\C[T(N-i), i\in\N]\setminus\{0\}$ for all $\ga\in J_w$. Define
$\Ht'(w)=\max\{\Ht(\gamma'):\gamma'\in J_{w}\}.$

\begin{lem}\label{K=-2submodule}
Suppose $\vphi(h(N))\neq0$ and $\vphi(K)=-2$. Then any submodule of $\hM$ is of the form $U(\hsl_2)Iu$, where $I$ is an ideal of $\C[T(N-i), i\in\N]$.
\end{lem}

\begin{proof}
Let $W$ be a nonzero submodule of $\hM$. By Lemma \ref{w in C[T]u}, any singular vector in $W$ lies in   $\C[T(N-i), i\in\N]u$. Let $I$ be an ideal of $\C[T(N-i), i\in\N]$ such that $Iu$ is the set of singular vectors in $W$. Note that $U(\hsl_2)Iu\subseteq W$. We will show that $W=U(\hsl_2)Iu$.

To the contrary, assume  $W\ne U(\hsl_2)Iu$.
Let $$H=\min\{\Ht'(v)\mid v\in W\setminus U(\hsl_2)Iu\}.$$
Take any nonzero weight vector $w\in W\setminus U(\hsl_2)Iu$ in the form of \eqref{w J_w} with minimal $H=\Ht'(w)$. Assume $P_{\gamma'}\notin I$ for all $\gamma'\in J_w$. Choose $\gamma\in J_w$ such that $\Ht(\gamma)=H$.
Without loss of generality, we can  assume that
$$l=\ell(\gamma)=\min\{\ell(\mu)\mid H=\Ht(\mu),\mu\in J_w\}.$$
We will get  a contradiction by inductions on $H$ and  on $l$.

If $H=0$ then $w=f(N)^m Pu$ for some $P\in \C[T(N-i), i\in\N]\setminus  I$ and $m\in\Z_+$, as $w$ is a weight vector. An easy calculation gives
$e(0)^m w=m!\varphi(h(N)) P  u\in W$, yielding  $P\in I$, which is a contradiction.

Now assume that $H>0$.
Suppose first that $\ga_-\neq\bo$. Write $\ga_-=(m_1,\dots,m_k)$. Set $\bar\ga_-=(m_1,\dots,m_{k-1})\in\Ga_N$ and $\bar\ga=(\ga_+,\bo,\bar\ga_-)\in\Ga^{(0)}$.
It is clear that $\Ht(\bar\ga)=H-N+m_k$ and $\ell(\bar\ga)=l-1$.

Note that
\begin{equation}\label{3.17e(N-m_k)}\aligned
	  e(N-m_k)&\sum_{\ga'\in J_w}X(\gamma')P_{\gamma'} u=\sum_{\ga'\in J_w}e(\ga'_+)e(N-m_k)f(\ga'_-)P_{\gamma'} u\\
	=&\sum_{\ga'\in J_w}e(\ga'_+)[e(N-m_k),f(\ga'_-)]P_{\gamma'} u.\\
	\endaligned\end{equation}
Take any $\ga'=(\ga'_+,\ga'_0,\ga'_-)\in J(H,l)_0\cap J_w$. By the PBW Theorem, we can write
\begin{equation}\label{3.17sum}\aligned
	&e(\ga'_+)[e(N-m_k),f(\ga'_-)]\\
	=&\sum_{\ga''\in I_1}b_{\ga''}X(\ga'')+\sum_{(\ga'',n)\in I_2}c_{\ga'',n}X(\ga'')\big(h(n)+\delta_{n,0}d_{\gamma''}K\big)+Y,
	\endaligned\end{equation}
where $I_1$ consists of elements $\ga''\in\Ga^{(0)}$ with $\ga''_+=\ga'_+, \ell(\ga''_-)=\ell(\ga'_-)-1$;
$I_2$ consists of elements $(\ga'',n)\in\Ga^{(0)}\times\Z$ with $\ga''_+=\ga'_+,\Ht(\ga'')\le H-N+m_k$ and
  $\ell(\ga'')=\ell(\ga')-1$;
  and $Y\in U(\widehat{\sl}_2)(e\otimes\C[t])+U(\widehat{\sl}_2)((\C h\oplus\C f)\otimes t^{N+1})$. Note that only finitely many $b_{\ga''}\in\C$ for $\ga''\in I_1$, or $c_{\ga'',n}\in\C$ for $(\ga'',n)\in I_2$, are nonzero.
  Clearly, $Y u=0$.

Consider the first summand in \eqref{3.17sum} and take $\gamma''\in I_1$. Suppose $\gamma'_-=(n_1,\dots,n_s)$. Then there is some $n_i,n_j\in\Z$ satisfying
$$\Ht(\gamma')-(2N-n_i-n_j)=\Ht(\gamma'')-(N-(n_i+n_j+N-m_k)).$$
Then $\Ht(\gamma'')=\Ht(\gamma')-2N+m_k\le H-2N+m_k$. So $\gamma''\neq \bar\gamma$ in this case.

Then we take $(\ga'',n)\in I_2$ with $n\geq 0$ and we have
$$\Ht(\gamma')-\Ht(\gamma'')=N-(n-(N-m_k))=2N-n-m_k.$$
Suppose $\gamma''=\bar\gamma$. Then we have $\Ht(\gamma')=H+N-n$. So $n=N$ and  $\gamma'=\gamma$.

For $(\ga'',n)\in I_2$ with $n<0$, we can repeatedly use (\ref{tau}) to substitute $h(n)$ and to write  $X(\gamma'')h(n)  u$ as a sum of terms of the form  $X(\gamma''')Q_{\gamma'''} u$, where $Q_{\gamma'''}\in \C[T(N-i), i\in\N]\setminus\{0\}$. Moreover, from \eqref{tau} and Lemma \ref{Ht property} we have
$$\Ht(\ga''')<\Ht(\ga'')+N-n=\Ht(\gamma')-N+m_k\leq H-N+m_k,$$
or $\Ht(\ga''')=\Ht(\gamma')-N+m_k,\ell(\ga''')=\ell(\gamma'')+2=\ell(\gamma')+1$.
If $\ga'''=\bar\gamma$, then $\Ht(\gamma')=H$. So $l-1=\ell(\ga''')=\ell(\ga')+1\geq l+1$, which is a contradiction.


Therefore the right-hand side in \eqref{3.17sum} has a nonzero projection
 to $$X(\bar\ga)\C[T(N-i), i\in\N]u$$ if and only if $\gamma'=\ga$, and the projection equals
$sX(\bar\ga)P_\ga u$, where $s$ is the number of $i$'s, $1\leq i\leq k$ with $m_i=m_k$.
As the result, $e(N-m_l)w$ has a nonzero projection to $X(\bar\ga)\C[T(N-i), i\in\N]u$, and in particular,
$e(N-m_k)w\neq 0$. Moreover, we have $\mu<H$ or $\mu=H$, $\ell(\mu)>l$ for all $\mu\in J_{e(N-m_k)w}\setminus\{\bar\gamma\}$.
Then we have $\Ht'(e(N-m_k)w)<H$ or $\Ht(\bar\gamma)=\Ht'(e(N-m_k)w)=H$,$\ell(\bar\gamma)=l-1$.
By the induction hypothesis we get $P(\ga)\in I$, which is a contradiction. The statement of the lemma follows in this case.

Now we suppose $\ga_+\neq\bo$, which is equivalent to $\Ht(\ga_+)>0$. Replacing $e(N-m_k)$ with $f(N-m')$ for some $m'<0$ and using similar arguments to the ones above,
we get again a contradiction. The result follows.
%
%
%
%
%
%
\end{proof}

For any integers $N_1,N_2$ with $N=N_1+N_2\geq 0$, and a Lie algebra homomorphism $\varphi: \S_{N_1,N_2}\to\C$ with $\varphi(K)=-2$ and  $\theta_i\in\C,i\in\N$, denote by $\widehat M(\theta_i,i\in\N)$ the submodule of $\widehat M(\varphi)$ generated by  singular vectors
$$(\sigma_{-N_1-1}(T(N-i))-\theta_i)u,i\in\N.$$
Let
$$\overline M(\varphi;\theta_i,i\in\N)=\widehat M(\varphi)/\widehat M(\theta_i, i\in\N)$$
 be the corresponding  quotient module.

\begin{thm}\label{irre K=-2 general}
Suppose $\vphi(h(N))\neq0$ and $\vphi(K)=-2$. Then any irreducible $\hsl_2$-subquotient of $\hM(\vphi)$
is isomorphic to a module of the form $\overline M(\vphi;\theta_i, i\in\N), \theta_i\in\C$.
\end{thm}
\begin{proof}
By (\ref{reduce}), it suffices to prove the statement for the case $N_1=-1$ and $N_2=N$. Assume that $\vphi=\vphi_{-1,N}$.

Let $V, M$ be two submodules of $\hM(\vphi)$ such that $V\subset M$ and $M/V$ is an irreducible $\hsl_2$-module. By the Schur's lemma, for any $i\in\N$, $T(N-i)$ acts on $M/V$ as some scalar $\theta_i\in \C$.
  By Lemma \ref{K=-2submodule}, $M=U(\hsl_2)Iu$ and $V=U(\hsl_2)Ju$ for some ideals $I,J$ of $\C[T(N-i),i\in\N]$. Clearly, $J\subsetneq I$. Let $v\in Iu\setminus Ju$. Then the map
$$f:U(\hsl_2)v\rightarrow \hM(\vphi),\ x\cdot v\mapsto x\cdot u,x\in U(\hsl_2)$$
is an isomorphism of $U(\hsl_2)$-modules, and we have the induced isomorphism
$$\bar f:U(\hsl_2)v/\langle(T(N-i)-\theta_i,i\in\N)v\rangle\rightarrow \hM(\vphi;\theta_i, i\in N).$$
Thus, $M/V$ and $\overline M(\vphi;\theta_i, i\in N)$ are isomorphic as $\hsl_2$-modules.
\end{proof}




\medskip

\subsection{Isomorphisms}
Fix  $N_1,N_2,N_1',N_2'\in\Z$ with $N_1+N_2\geq0$ and $N_1'+N'_2\geq0$, $\theta_i,\theta'_i\in\C$ for $i\in\N$,
and Lie algebra homomorphisms $\vphi=\vphi_{N_1,N_2}: \S_{N_1,N_2}\to\C$ and $\vphi'=\vphi'_{N_1,N_2}: \S_{N'_1,N'_2}\to\C$.
In this section we determine the isomorphisms between $\hM(\varphi)$ and $\hM(\vphi')$.

\begin{thm}\label{iso hg}
 We have
\begin{itemize}
\item[(1)] $\hM(\varphi)\cong\hM(\vphi')$ if and only if $N_1=N_1'$,  $N_2=N'_2$ and $\vphi=\vphi'$.
\item[(2)] If $\varphi(K)=\varphi'(K)=-2$, then
$$\overline M(\varphi;\theta_i,i\in\N)\cong\overline M(\varphi';\theta'_i,i\in\N)$$
if and only if $N_1=N_1',N_2=N_2',\varphi=\varphi'$, and $\theta_i=\theta_i'$ for any $i\in\N$.
\item[(3)] If $\varphi'(K)=-2$, then $\hM(\varphi)\ncong\overline M(\varphi';\theta_i,i\in\N)$.
\end{itemize}
\end{thm}

\begin{proof}

(1) We first assume $N_1>N_1'$. Clearly,
$e(N_1-1)$ annihilates  some nonzero vectors in $\hM(\vphi')$ but does not
annihilate  any nonzero vectors in $\hM(\vphi)$. Hence
$\hM(\vphi)\not\cong\hM(\vphi')$.
Similarly, we have $\hM(\vphi)\not\cong\hM(\vphi')$ if $N_2\neq N_2'$.

Now suppose $N_1=N_1'$ and  $N_2=N'_2$. By comparing the nonzero vectors of length $0$ in $\hM(\vphi)$ and $\hM(\vphi')$, as well as the action of $\S_{N_1,N_2}$ on them, we get $\varphi=\varphi'$.

(2) Similarly to the proof of (1) we can show that $N_1=N_1'$ and $N_2=N_2'$ if $\overline M_(\varphi;\theta_i,i\in\N)\cong\overline M(\varphi';\theta'_i,i\in\N)$. By considering the action of $T(N-i)$ we see that $\theta_i=\theta'_i$ for all $i\in\N$. By comparing the Whittaker  vectors in $\overline M(\varphi;\theta_i,i\in\N)$ and $\overline M(\varphi';\theta_i,i\in\N)$, as well as the action of $\S_{N_1,N_2}$ on them, we get $\varphi=\varphi'$.

(3) Suppose $\hM(\varphi)\cong\overline M(\varphi';\theta_i,i\in\N)$. By considering the action of $K$ we see that $\varphi(K)=\varphi'(K)=-2$. Then (3) is clear since  $\hM(\varphi)$ is not an irreducible $\hsl_2$-module, while $\overline M(\varphi';\theta_i,i\in\N)$ is irreducible.\end{proof}



%
%
%

\section{Smooth modules over $\tg$}\label{section tg}

\subsection{Irreducibility of smooth $\tg$-modules}
We use notation from Section 2. In particular, recall  the  induced $\tg$-module
$$\tM(\varphi_{N_1,N_2})=\C[d]\ot\hM(\vphi_{N_1,N_2}).$$
Recall also that the action of the Casimir element $\O$ in \eqref{casimir} on $\tM(\vphi_{N_1,N_2})$
is well-defined and commutes with any elements in $U(\tg)$.
 In particular, the action of $\Omega$ on the singular vector $1\ot u\in \tM(\vphi_{N_1,N_2})$ can be computed explicitly as follows
\begin{equation}\aligned\label{Omega}
&\Omega(1\ot u) =\\
=&2d\ot(\vphi(K)+2)u + 1 \ot \left(\frac{1}{2}h(0)^2 +h(0) +2f(0)e(0)\right)u\\
 &+2\ot\left(\sum_{i=1}^{N_2}e(-i)f(i)+\sum_{i=1}^{N_1}f(-i)e(i) +\hskip-3pt\sum_{i=1}^{N_1+N_2+1}\hskip-3pt h(-i)h(i)\right)u.\\
\endaligned\end{equation}

We will denote for simplicity $\vphi=\vphi_{N_1,N_2}$.
For any polynomial $P\in\C[\Omega]$ we have the $\tg$-submodule $P(\Omega)\tM(\vphi)$ of
$\tM(\vphi)$. 

\begin{thm}\label{irre tM} Suppose that $\vphi(N_1+N_2+1)\neq0$ and $\vphi(K)\neq-2$.
For any nonzero  $P\in\C[\Omega]$ and $\l\in\C$, the quotient
$$\tM(\vphi,\l):=P(\Omega)\tM(\vphi)/(\Omega-\l)P(\Omega)\tM(\vphi)$$
is an irreducible $\hg$-module,  independent of the choice of $P$.
\end{thm}

\begin{proof}
Fix any nonzero $P\in\C[\O]$ and set $\tM=\tM(\vphi)$ for short.
Define a linear map
$$\phi:\ \tM/(\O-\l)\tM\to P(\Omega)\tM/(\Omega-\l)P(\Omega)\tM$$
by $\phi(x(1\ot u))=xP(\O)(1\ot u)$ for any $x\in U(\tg)$. To show that $\phi$ is an isomorphism,
we need only to show that $\tM/(\O-\l)\tM$ is irreducible, since $\phi$ is surjective.

Note that $\tM/(\O-\l)\tM\cong\hM(\vphi)$ as $\hg$-modules by (\ref{Omega}), since $\vphi(K)\neq-2$.
By Theorem \ref{irre general},  the $\hg$-moduyle $\hM(\vphi)$ is irreducible, which  implies the irreducibility of $\tM/(\O-\l)\tM$ over
$\hg$ and hence  over $\tg$.
\end{proof}

On the other hand, if $\vphi(K)=-2$, we can form the induced $\tg$-module
$$\tM(\varphi;\theta_i, i\in\N)=\Ind_{\hg}^{\tg}\bM(\vphi;\theta_i,i\in\N),$$
 which is isomorphic to $\C[d]\ot\bM(\vphi;\theta_i,i\in\N)$
as a vector space.
Set $\bM=\bM(\vphi;\theta_i,i\in\N)$ and $\tM=\tM(\varphi;\theta_i, i\in\N)$ for short. For convenience, we also write $d^iw=d^i\ot w$ for any $w\in \bM$.

\begin{thm}\label{irre tM K=-2}
Suppose $\vphi(N_1+N_2+1)\neq0$ and $\vphi(K)=-2$.
Then $\tM(\varphi;\theta_i, i\in\N)$
is an irreducible $\tg$-module.
\end{thm}

\begin{proof} We only prove the case  $N_1=-1$ and $N_2=N\geq1$, since other cases can be considered similarly.
Set $\tM^{(k)}=\sum_{i=0}^kd^i\bM$
and let $W$ be a nonzero submodule of $\tM$. Take any $w=\sum_{i=0}^kd^i w_i\in W\setminus\{0\}$, where $w_i\in\bM$ and $w_k\neq0$ with minimal $k$. It is enough to show that $k=0$. To the contrary, assume $k\geq1$.
 There exists $x\in U(\g\ot\C[t^{\pm1}])$ such that $xw_k=\bar u$. Then we have
$$xw=\sum_{i=0}^kx(d^i w_i)\equiv d^kxw_k\mod \tM^{(k-1)}.$$
Replacing $w$ with $xw$, we may assume that $w_k=\bar u$.

If $w_{k-1}$ is a singular vector in the $\hg$-module $\bM$, then $w_{k-1}$ is a scalar multiple of $\bar u$ (see the proof of Theorem \ref{irre K=-2}), and we have the following vector in $W$:
$$\aligned
(h(N)-&\vphi(h(N)))w\equiv [h(N),d^k]u+[h(N),d^{k-1}]w_{k-1}\\
\equiv& -kNd^{k-1}h(N)u=-kN\vphi(h(N))d^{k-1}u\mod\tM^{(k-2)},
\endaligned$$
which is a contradiction.

So $w_{k-1}$ is not a singular vector in  $\bM$. Then, from the proof of Theorem \ref{irre K=-2} we get that
either $e(m)w_{k-1}\neq0$ for some $m\in\Z_+$, or $f(m')w_{k-1}\neq0$ for some $m'\in\Z, m'>N$.
In the former case we have
$$\aligned
e(m)w\equiv &e(m)d^ku+e(m)d^{k-1}w_{k-1}\equiv[e(m),d^k]u+d^{k-1}e(m)w_{k-1}\\
\equiv& -kmd^{k-1}e(m)u+d^{k-1}e(m)w_{k-1}\\
=&d^{k-1}e(m)w_{k-1}\mod \tM^{(k-2)},\\
\endaligned$$
which is a contradiction.
Now, in the latter case we have
$$\aligned
f(m')w\equiv &f(m')d^ku+f(m')d^{k-1}w_{k-1}=[f(m'),d^k]u+d^{k-1}f(m')w_{k-1}\\
\equiv& -km'd^{k-1}f(m')u+d^{k-1}f(m')w_{k-1}\\
=&d^{k-1}f(m')w_{k-1}\mod \tM^{(k-2)},\\
\endaligned$$
which is again a contradiction.
The lemma follows.
\end{proof}

As a corollary, we can give a description of singular vectors in $\tM$.
A vector in the $\tg$-module $\tM$ is called a singular vector if it is a singular vector
 of $\tM$ as a $\hg$-module.

\begin{prop}\label{singular tM}
Suppose $\vphi(N_1+N_2+1)\neq0$ and $\vphi(K)=-2$.
The set of singular vectors of $\tM$ coincides with
$\C u$.
\end{prop}

\begin{proof} Let $w=\sum_{i=0}^kd^i w_i$ be a nonzero singular vector in $\tM(\varphi;\theta_i, i\in\N)$, where $w_i\in\bM$ and $w_k\neq0$.

Suppose first $k\geq1$.
Note that $w_k$ is not a singular vector in the $\hg$-module $\bM$
by the proof of Theorem  \ref{irre K=-2}.
Without loss of generality
we may assume that $w_k$ is a weight vector. Then there exists
$$x\in \{e(N_1+m), f(N_2+m), h(m)-\vphi(h(m)), m\in\N\}$$
such that $xw_k\neq0$. We have
$$xw\equiv xd^kw_k\equiv d_kxw_k\neq0 \mod\sum_{i=0}^{k-1}d^i\bM.$$
We get a contradiction.

Now assume $k=0$. Then $w\in\bM$ is a singular vector in the $\tg$-module $\tM$ if and only if it is a
singular vector in the $\hg$-module $\bM$. Hence the result follows from
Theorem \ref{irre K=-2}.
\end{proof}

%
%
%
%
%
%
\subsection{Isomorphisms} Fix any $N_1,N_2,N_1',N_2'\in\Z$ with $N_1+N_2\geq0$ and $N_1'+N'_2\geq0$,
and Lie algebra homomorphisms $\vphi: \S_{N_1,N_2}\to\C$ and $\vphi': \S_{N'_1,N'_2}\to\C$.
We  determine the isomorphisms between constructed $\tg$-modules.

\begin{thm}
In the notation above we have the following.
\begin{itemize}
\item[(1)] Suppose $\vphi(K)\neq-2$, $\vphi'(K)\neq-2$ and take any $\l,\l'\in\C$. Then $\tM(\vphi,\l)\cong\tM(\vphi',\l')$ if and only if $N_1=N_1'$,  $N_2=N'_2$, $\vphi=\vphi'$ and $\l=\l'$.
\item[(2)] Suppose $\vphi(K)=-2$, $\vphi'(K)=-2$ and take $\theta_i,\theta_i'\in\C, i\in\N$.
Then $\tM(\vphi;\theta_i,i\in\N)\cong\tM(\vphi';\theta'_i,i\in\N)$ if and only if $N_1=N_1'$,  $N_2=N'_2$, $\vphi=\vphi'$ and $\theta_i=\theta_i'$ for all $i\in\N$.
\item[(3)] Suppose $\vphi(K)\neq-2$, $\vphi'(K)=-2$ and take $\l,\theta_i'\in\C, i\in\N$.
Then $\tM(\vphi;\l)\not\cong\tM(\vphi';\theta'_i,i\in\N)$.
\end{itemize}
\end{thm}

\begin{proof}
Assertion (1) follows from the fact that isomorphic $\tg$-modules have the same central character with respect to the action of the Casimir element $\Omega$ and are also isomorphic as $\hg$-modules, thanks to the first assertion of Theorem \ref{iso hg}.

To prove (2), assume $\tM(\vphi;\theta_i,i\in\N)\cong\tM(\vphi';\theta'_i,i\in\N)$ and let $u, u'$ be the canonical generating vectors in the respective $\hg$-modules. By Proposition \ref{singular tM},
the sets of singular vectors in these two modules are $\C\bar u$ and $\C\bar u'$ respectively.
Then the generating vector $\bar u$ must correspond to $\bar u'$ up to a scalar under the above isomorphism,
which in turn induces an isomorphism between the $\hg$-modules $\bM(\vphi;\theta_i,i\in\N)$
and $\bM(\vphi';\theta'_i,i\in\N)$. The result then follows from Theorem \ref{iso hg}.

Assertion (3) is clear, since $\tM(\vphi;\l)$ is irreducible over $\hg$ while\\
$\tM(\vphi';\theta'_i,i\in\N)$ is not.
\end{proof}

\medskip

\subsection{More $\tg$-modules}

Let $E({\bf a},{\bf\l})$ be the irreducible $\hg$-module from \cite{CP2} (see also  \cite[Section 3]{GZ}). Then \cite[Theorem 3.4]{GZ} implies that
 the $\tg$-modules
$$(E({\bf a},{\bf\l})\ot \hM(\vphi_{N_1,N_2}))[d]$$
are  irreducible if
 $\vphi(N_1+N_2+1)\neq0$ and $\vphi(K)\neq-2$.
 Similarly, the $\tg$-modules
$$(E({\bf a},{\bf\l})\ot \bM(\vphi_{N_1,N_2};\theta_i,i\in\N))[d]$$
are  irreducible if
 $\vphi(N_1+N_2+1)\neq0$ and $\vphi(K)=-2$.

\medskip

\section{ Acknowledgments}
\noindent V. Futorny is partially supported by NSF of China (12350710787 and 12350710178); X. Guo is partially supported by NSF of China (Grant 11971440) and NSF of Guangzhou University (RC2023062); Y. Xue is partially supported by NSF of China (12301037); K. Zhao is partially supported by NSERC (311907-2020).

\vspace{0.2cm} \noindent V.F.: Shenzhen International Center for Mathematics, Southern University of Science and
Technology, Shenzhen, China.\\
Email: vfutorny@gmail.com

\vspace{0.4cm} \noindent X.G.: School of Mathematics and Information Science, Guangzhou University, Guangzhou
510006, China. \\
Email: guoxq@gzhu.edu.cn.

\vspace{0.4cm} \noindent
Y. X.: School of Mathematics and Statistics, Nantong University, Nantong, Jiangsu, 226019, P. R. China.\\
Email: yxue@ntu.edu.cn

\vspace{0.2cm} \noindent K.Z.: College of
Mathematics and Information Science, Hebei Normal (Teachers)
University, Shijiazhuang, Hebei, 050016 P. R. China, and Department of Mathematics, Wilfrid
Laurier University, Waterloo, ON, Canada N2L 3C5. \\
Email: kzhao@wlu.ca


\begin{thebibliography}{99}










\bibitem [ALZ] {ALZ} D. Adamovic, R. Lu, K. Zhao, Whittaker modules for the affine Lie algebra $A^{(1)}_1$,   Adv. Math. 289 (2016), 438--479.


\bibitem[BBFK]{BBFK} V. Bekkert, G. Benkart, V. Futorny, I. Kashuba,  New simple modules for Heisenberg and affine Lie algebras, J. Algebra 373 (2013), 284-298.

\bibitem[CF]{CF} M. C. Cardoso, V. Futorny, Affine Lie algebra representations induced from Whittaker modules, Proceedings AMS, 151(7), 2023, 1041-1053.

 \bibitem[Ch]{Ch} V. Chari, Integrable representations of affine Lie-algebras, Invent. Math. 85 (1986), 317-335.

\bibitem[CP1]{CP} V. Chari and A. Pressley, A guide to quantum groups. Corrected
reprint of the 1994 original. Cambridge University Press,
Cambridge, 1995. xvi+651 pp.

\bibitem[CP2]{CP1} V. Chari and A. Pressley, New unitary representations of loop groups, Math. Ann. 275 (1986), 87-104.

\bibitem[CP3]{CP2}  V. Chari and A. Pressley, Integrable representations of twisted affine Lie algebras, J. Algebra 113 (1988), 438--464.

\bibitem[CP4]{CP3} V. Chari and A. Pressley, A new family of simple, integrable modules for affine Lie
algebras, Math. Ann. 275 (1986), no. 1, 87-104.


\bibitem[Chr]{Chr} K. Christodoulopoulou, Whittaker modules for Heisenberg algebras and imaginary Whittaker modules for affine Lie algebras, J. Algebra {\bf 320} (2008) 2871-2890.



\bibitem[CFu]{CFu} B. Cox, V. Futorny, Intermediate Wakimoto modules for affine $sl(n + 1)$, J. Physics A: Math. Gen. 37 (2004), 5589-5603.


\bibitem[DG]{DG} I.Dimitrov, D.Grantcharov, Classification of simple weight modules over affine Lie algebras, arXiv:0910.0688.



\bibitem[E]{E} S. Eswara Rao, Classification of loop modules with finite-dimensional weight spaces,
Math. Ann. 305 (1996), 651-663.


\bibitem[FMS]{FMS} P. Di Francesco, P. Mathieu and D. Senechal, Conformal Field Theory, Springer Verlag, New York (1997).

\bibitem[F]{F1} V. Futorny, Imaginary Verma modules for affine Lie algebras, Canad. Math. Bull. 37 (1994) 213-218.

\bibitem[F2]{F2}V. Futorny, Simple non-dense $A^{(1)}_1$-modules, Pacific J. Math. 172 (1996) 83-99.

\bibitem[F3]{F3} V. Futorny, Representations of Affine Lie Algebras, Queen's Papers in Pure and Appl. Math., vol. 106, Queen's University,
Kingston, ON, 1997.



\bibitem[FK1]{FK1} V. Futorny, I.Kashuba, Structure of parabolically induced modules for affine Kac-Moody algebras, Journal of Algebra 500 (2018), 362-374.

\bibitem[FK2]{FK2} V. Futorny, I.Kashuba, Induced modules for Affine Lie algebras, SIGMA: Symmetry, Integrability and Geometry: Methods and Applications, 5 (2009), 026.

\bibitem[FKr]{FKr} V. Futorny,  L.Krizka, Positive energy representations of affine vertex alge- bras, Communications. Math. Physics, 383 (2021), 841-891.

\bibitem[FT]{FT} V. Futorny, A. Tsylke, Classification of irreducible nonzero level modules with finite--dimensional weight spaces for affine Lie algebras, J. Algebra 238 (2001) 426-441.




   \bibitem[GZ]{GZ} X. Guo, K. Zhao, Simple representations of untwisted  affine Kac-Moody algebras, arXiv:1305.4059v2.




\bibitem[K]{K}  V. Kac, {Infinite dimensional Lie algebras, 3rd edition},
Cambridge Univ. Press, 1990.



\bibitem[Li]{Li} H. Li, On certain categories of modules for affine Lie algebras, Math. Z. {\bf 248} (2004), no. 3, 635-664.

\bibitem[LL]{LL} J. Lepowsky and H. Li, Introduction to vertex operator algebras and their representations. Progress in Mathematics, 227. Birkhauser Boston, Inc., Boston, MA, 2004.


\bibitem[Lu]{Lu} G. Lusztig, Quantum deformations of certain simple modules over
enveloping algebras. Adv. Math., {\bf 70} (1988), no. 2, 237--249.

\bibitem[MZ]{MZ}V. Mazorchuk, K. Zhao, Characterization of simple highest weight modules,
Canadian Mathematical Bulletin, {\bf 56} (2013), n.3,   606-614.





\bibitem[NXZ]{NXZ} K. Nguyen, Y. Xue, K. Zhao, A class of new simple smooth modules over the affine algebra $A_1^{(1)}$, preprint, August, 2023.




\bibitem[Wi]{Wi} B. J. Wilson, Highest-weight theory for truncated current Lie algebras, J. Algebra {\bf 336} (2011) 1--27.

\end{thebibliography}
\end{document}